\documentclass{amsart}

\usepackage[]{amsmath}
\usepackage{amsfonts}
\usepackage{amssymb}
\usepackage{mathrsfs}

\usepackage[all,knot]{xy}
\xyoption{arc}

\newtheorem{theorem}{Theorem}[section]
\newtheorem{lemma}[theorem]{Lemma}

\theoremstyle{definition}

\newtheorem{proposition}[theorem]{Proposition}
\newtheorem{corollary}[theorem]{Corollary}
\theoremstyle{remark}

\numberwithin{equation}{section}

\begin{document}

\title[]{A similarity invariant of a class of
$n$-normal operators in terms of $K$-theory}

\author[C. Jiang]{Chunlan Jiang}
\address{Chunlan Jiang\\
Department of Mathematics, Hebei Normal University, Hebei,
050024, P. R. China} \email{cljiang@hebtu.edu.cn}


\author[R. Shi]{Rui Shi$^{\ast}$}


\address{Rui Shi\\
}

\curraddr{
}
\email{ruishi.math@gmail.com}
\thanks{$^{\ast}${Corresponding author}.}





\begin{abstract}
In this paper, we prove an analogue of the Jordan canonical form
theorem for a class of $n$-normal operators on complex separable
Hilbert spaces in terms of von Neumann's reduction theory. This is a
continuation of our study of bounded linear operators, the
commutants of which contain bounded maximal abelian set of
idempotents. Furthermore, we give a complete similarity invariant
for this class of operators by $K$-theory for Banach algebras.
\end{abstract}

\maketitle

\section{Introduction} 

In this paper the authors continue the study on generalizing the
Jordan canonical form theorem for bounded linear operators on
separable Hilbert spaces, which was initiated in \cite{Shi_1} and
carried on in \cite{Shi_2}. Throughout this article, we only discuss
Hilbert spaces which are {\textit{complex and separable}}. Denote by
$\mathscr{L}(\mathscr{H})$ the set of bounded linear operators on a
Hilbert space $\mathscr{H}$. An \textit{idempotent} $P$ on
$\mathscr{H}$ is an operator in $\mathscr{L}(\mathscr{H})$ such that
$P^{2}=P$. A {\textit{projection}} $Q$ in $\mathscr{L}(\mathscr{H})$
is an idempotent such that $Q=Q^{*}$. An operator $A$ in $\mathscr
{L}(\mathscr {H})$ is said to be {\textit{irreducible}} if its
commutant $\{A\}'\triangleq\{B\in\mathscr {L}(\mathscr {H}):AB=BA\}$
contains no projections other than $0$ and the identity operator $I$
on $\mathscr{H}$, introduced by P. Halmos in \cite{Halmos}. (The
separability assumption is necessary because on a nonseparable
Hilbert space every operator is reducible.) An operator $A$ in
$\mathscr {L}(\mathscr {H})$ is said to be {\textit{strongly
irreducible}} if $XAX^{-1}$ is irreducible for every invertible
operator $X$ in $\mathscr {L}(\mathscr {H})$, introduced by F.
Gilfeather in \cite{Gilfeather}. This shows that the commutant of a
strongly irreducible operator contains no idempotents other than $0$
and $I$. We observe that strong irreducibility stays invariant up to
similar equivalence while irreducibility is an invariant up to
unitary equivalence. For an operator $A$ in $\mathscr {L}(\mathscr
{H})$, a nonzero idempotent $P$ in $\{A\}^{\prime}$ is said to be
{\textit{minimal}} if every idempotent $Q$ in
$\{A\}^{\prime}\cap\{P\}^{\prime}$ satisfies $QP=P$ or $QP=0$. For a
minimal idempotent $P$ in $\{A\}^{\prime}$, the restriction
$A|^{}_{\textrm{ran} P}$ is strongly irreducible on $\textrm{ran}
P$. For $n$ in $\mathbb {N}\cup\{\infty\}$, we write $\mathscr
{H}^{(n)}_{}$ for the orthogonal direct sum of $n$ copies of
$\mathscr {H}$, where we denote by $\mathbb {N}$ the set of positive
integers. For an operator $T$ in $\mathscr {L}(\mathscr {H})$ and
$n$ in $\mathbb {N}\cup\{\infty\}$, the orthogonal direct sum of $T$
with itself $n$ times is denoted by $T^{(n)}_{}$. Let $\mathscr {T}$
be a subset of $\mathscr {L}(\mathscr {H})$. Then we write $\mathscr
{T}^{(n)}_{}$ for $\{T^{(n)}\in\mathscr {L}(\mathscr
{H}^{(n)}_{}):T\in\mathscr {T}\}$ and $\mathscr {T}^{\prime}$ for
the commutant of $\mathscr {T}$.

On a finite dimensional Hilbert space $\mathscr{K}$, the Jordan
canonical form theorem shows that every operator $B$ in $\mathscr
{L}(\mathscr {K})$ can be uniquely written as a (Banach) direct sum
of Jordan blocks up to similarity. An important observation is that
for any two bounded maximal abelian sets of idempotents
$\mathscr{Q}$ and $\mathscr{P}$ in the commutant $\{B\}^{\prime}$,
there exists an invertible operator $X$ in $\{B\}^{\prime}$ such
that
$$X\mathscr{Q}X^{-1}=\mathscr{P}.\eqno{(1.1)}$$ Thus, we obtain
$K^{}_{0}(\{B\}^{\prime})\cong\mathbb{Z}^{(k)}_{}$ and
$K^{}_{1}(\{B\}^{\prime})\cong{0}$ by a routine computation, where
we let $k$ denote the number of minimal idempotents in
$\mathscr{P}$. Furthermore, the ordered $K^{}_{0}$ groups can be
viewed as a complete similarity-invariant in the following way. Let
$B^{}_{1}$ and $B^{}_{2}$ be in $\mathscr {L}(\mathscr {K})$ such
that
$$\theta^{}_{1}(K^{}_{0}(\{B^{}_{1}\}^{\prime}))=\theta^{}_{2}(K^{}_{0}(\{B^{}_{1}\oplus
B^{}_{2}\}^{\prime}))=\mathbb{Z}^{(k)}_{},\eqno{(1.2)}$$ and
$\theta^{}_{1}([I^{}_{\{B^{}_{1}\}^{\prime}}])=n^{}_{1}e^{}_{1}+\cdots+n^{}_{k}e^{}_{k}$
and $\theta^{}_{2}([I^{}_{\{B^{}_{1}\oplus
B^{}_{2}\}^{\prime}}])=2n^{}_{1}e^{}_{1}+\cdots+2n^{}_{k}e^{}_{k}$,
where $\theta^{}_{1}$ and $\theta^{}_{2}$ are group isomorphisms
essentially induced by the standard traces of matrices and
$\{e^{}_{i}\}^{k}_{i=1}$ are the generators of the semigroup
$\mathbb{N}^{(k)}_{}$ of $\mathbb{Z}^{(k)}_{}$ and
$I^{}_{\{B^{}_{1}\}^{\prime}}$ is the unit of
$\{B^{}_{1}\}^{\prime}$, then $B^{}_{1}$ is similar to $B^{}_{2}$.
The reader is referred to Chapter $2$ of \cite{Jiang_5} for the
details skipped above.

In our first attempt to prove an analogue of the Jordan canonical
form theorem in \cite{Shi_1}, we observe that minimal idempotents in
$\{A\}^{\prime}$, for $A\in\mathscr {L}(\mathscr {H})$, play an
important role in the construction of the Jordan canonical form of
$A$. However, for a single self-adjoint generator $N$ of a diffuse
masa, the commutant $\{N\}^{\prime}$ contains no minimal
idempotents. This fact shows us that, on considering a
generalization of the Jordan canonical form theorem, direct sums of
Jordan blocks need to be replaced by direct integrals of strongly
irreducible operators with regular Borel measures to represent
certain operators in $\mathscr {L}(\mathscr {H})$.

We briefly introduce some concepts in the von Neumann's reduction
theory that will be employed in this paper. For the most part, we
follow \cite{Azoff_2,Schwartz}. Once and for all, Let $\mathscr
{H}^{}_{1}\subset\mathscr {H}^{}_{2}\subset\cdots\subset\mathscr
{H}^{}_{\infty}$ be a sequence of Hilbert spaces with $\mathscr
{H}^{}_{n}$ having dimension $n$ and $\mathscr {H}^{}_{\infty}$
spanned by the remaining $\mathscr {H}^{}_n$s. Let $\mu$ be (the
completion of) a finite positive regular Borel measure supported on
a compact subset $\Lambda$ of $\mathbb {R}$. (We realize this by
virtue of (\cite{Rosenthal}, Theorem $7.12$).) And let
$\{\Lambda^{}_{\infty}\}\cup\{\Lambda^{}_{n}\}^{\infty}_{n=1}$ be a
Borel partition of $\Lambda$. Then we form the associated direct
integral Hilbert space $$\mathscr {H}=\int^{\oplus}_{\Lambda}
\mathscr {H}(\lambda) d\mu(\lambda) \eqno{(1.3)}$$ which consists of
all (equivalence classes of) measurable functions $f$ and $g$ from
$\Lambda$ into $\mathscr {H}^{}_{\infty}$ such that
\begin{enumerate}
\item[(1)] $f(\lambda)\in \mathscr {H}(\lambda)\equiv \mathscr
{H}^{}_{n}$
for
$\lambda\in \Lambda^{}_{n}$;
\item[(2)] $\|f\|^{2}\triangleq\int^{}_{\Lambda}\|f(\lambda)\|^{2}
d\mu(\lambda)<\infty$;
\item[(3)] $(f,g)\triangleq\int^{}_{\Lambda}(f(\lambda),g(\lambda))
d\mu(\lambda)$.
\end{enumerate}
The element in $\mathscr {H}$ represented by the function
$\lambda\rightarrow f(\lambda)$ is denoted by
$\int^{\oplus}_{\Lambda}f(\lambda) d\mu(\lambda)$. An operator $A$
in $\mathscr {L}(\mathscr {H})$ is said  to be
{\textit{decomposable}} if there exists a strongly $\mu$-measurable
operator-valued function $A(\cdot)$ defined on $\Lambda$ such that
$A(\lambda)$ is an operator in $\mathscr {L}(\mathscr {H}(\lambda))$
and $(Af)(\lambda)=A(\lambda)f(\lambda)$, for all $f\in \mathscr
{H}$. We write
$A\equiv\int^{\oplus}_{\Lambda}A(\lambda)d\mu(\lambda)$ for the
equivalence class corresponding to $A(\cdot)$.  If $A(\lambda)$ is a
scalar multiple of the identity on $\mathscr {H}(\lambda)$ for
almost all $\lambda$, then $A$ is said to be {\textit{diagonal}}.
The collection of all diagonal operators is said to be the
{\textit{diagonal algebra}} of $\Lambda$. It is an abelian von
Neumann algebra. A decomposable operator $A$ in $\mathscr
{L}(\mathscr {H})$ is essentially a direct sum of $n$-normal
operators with respect to $n$. Let $\Lambda=\Lambda^{}_{n}$ and $A$
in $\mathscr {L}(\mathscr {H})$ be decomposable, then $A$ is
$n$-normal. An operator $A$ in $\mathscr {L}(\mathscr {H})$ is said
to be {\textit{$n$-normal}}, if there exists a unitary operator $U$
from $\mathscr {H}$ to $(L^{2}(\nu))^{(n)}$ such that
$$UAU^{*}=\begin{pmatrix}
M^{}_{f^{}_{11}}&\cdots&M^{}_{f^{}_{1n}}\\
\vdots&\ddots&\vdots\\
M^{}_{f^{}_{n1}}&\cdots&M^{}_{f^{}_{nn}}\\
\end{pmatrix}^{}_{n\times n}
\begin{matrix}
L^{2}(\nu)\\
\vdots\\
L^{2}(\nu)\\
\end{matrix} \eqno{(1.4)}$$ where $\nu$ is a finite positive regular
Borel measure supported on a compact subset $\Gamma$ of $\mathbb
{C}$ and $M^{}_{f^{}_{ij}}$ is a Multiplication operator for
$f^{}_{ij}$ in $L^{\infty}(\nu)$ and $1\leq i,j\leq n$. In the sense
of direct integral decomposition, the operator $UAU^{*}$ is in the
form
$$UAU^{*}=\int^{\oplus}_{\Gamma}\begin{pmatrix}
f^{}_{11}(\lambda)&\cdots&f^{}_{1n}(\lambda)\\
\vdots&\ddots&\vdots\\
f^{}_{n1}(\lambda)&\cdots&f^{}_{nn}(\lambda)\\
\end{pmatrix}d\nu(\lambda). \eqno{(1.5)}$$ Furthermore, by virtue of
(\cite{Azoff_1}, Corollary $2$), for every $n$-normal operator $A$
on $(L^{2}(\nu))^{(n)}$ and positive integer $n<\infty$, there
exists an $n$-normal unitary operator $U$ on $(L^{2}(\nu))^{(n)}$
such that $UAU^{*}$ is an upper triangular $n$-normal operator, i.e.
$$UAU^{*}=\begin{pmatrix}
M^{}_{f^{}_{11}}&M^{}_{f^{}_{12}}&\cdots&M^{}_{f^{}_{1n}}\\
0&M^{}_{f^{}_{22}}&\cdots&M^{}_{f^{}_{2n}}\\
\vdots&\vdots&\ddots&\vdots\\
0&0&\cdots&M^{}_{f^{}_{nn}}\\
\end{pmatrix}^{}_{n\times n}\begin{matrix}
L^{2}(\nu)\\
L^{2}(\nu)\\
\vdots\\
L^{2}(\nu)\\
\end{matrix}. \eqno{(1.6)}$$ The following two
basic results will be used in the sequel:
\begin{enumerate}
\item An operator acting on a direct integral of Hilbert spaces is
decomposable if and only if it commutes with the corresponding
diagonal algebra (\cite{Schwartz}, p. $22$).
\item Every abelian von Neumann algebra is (unitarily equivalent to)
an essentially unique diagonal algebra (\cite{Schwartz}, p. $19$).
\end{enumerate}

By the above observation of the Jordan canonical form theorem, our
first question is whether the commutant $\{A\}^{\prime}$ contains a
bounded maximal abelian set of idempotents for every operator $A$ in
$\mathscr {L}(\mathscr {H})$. In \cite{Shi_1}, we gave a negative
answer by constructing two operators $A$ and $B$ in the forms
$$A=\bigoplus\limits^{\infty}_{i=1}A^{}_{i},\,
A^{}_{i}=\begin{pmatrix}
\frac{1}{i}&1\\
0&-\frac{1}{2i}
\end{pmatrix}\in M^{}_{2}(\mathbb{C}),\mbox{ and }
B=\begin{pmatrix}
N^{}_{\mu}&I\\
0&-\frac{1}{2}N^{}_{\mu}\end{pmatrix}
\begin{matrix}
L^2(\mu)\\
L^2(\mu)\\
\end{matrix}, \eqno{(1.7)}$$
where the multiplication operator $N^{}_{\mu}$ is defined on
$L^{2}(\mu)$ by $N^{}_{\mu}f=z\cdot f$ for each $f$ in $L^{2}(\mu)$
and $\mu$ is a finite regular Borel measure supported on a compact
subset of $\mathbb {C}$. (And it is well known that the normal
operator $N^{}_{\mu}$ is star-cyclic.) Furthermore, we proved the
following theorem:

\begin{theorem}[\cite{Shi_1}, Theorem 1.2]
An operator $A$ in $\mathscr {L}(\mathscr {H})$ is similar to a
direct integral of strongly irreducible operators if and only if its
commutant $\{A\}^{\prime}$ contains a bounded maximal abelian set of
idempotents.
\end{theorem}

That {\textit{an operator $A$ is similar to a direct integral of
strongly irreducible operators}} denoted by
$XAX^{-1}=\int^{}_{\Lambda}B(\lambda)d\mu(\lambda)$ for some
invertible operator $X$ in $\mathscr {L}(\mathscr {H})$ means that
the Hilbert space $\mathscr {H}=\int^{\oplus}_{\Lambda} \mathscr
{H}^{}_{}(\lambda)d\mu(\lambda)$ is in the sense of $(1.3)$ and the
operator $XAX^{-1}$ is decomposable with respect to the
corresponding diagonal algebra such that the integrand $B(\cdot)$ is
a bounded strongly $\mu$-measurable operator-valued function defined
on $\Lambda$ and $B(\lambda)$ is strongly irreducible on the
corresponding fibre space $\mathscr{H}^{}_{}(\lambda)$ for almost
every $\lambda$ in $\Lambda$. For related concepts and results about
von Neumann's reduction theory, the reader is referred to
\cite{Azoff_2,Conway_1,Davidson,Rosenthal,Schwartz}.

Since then, we have paid more attention to the subset $\mathscr{S}$
of $\mathscr {L}(\mathscr {H})$, where the set $\mathscr{S}$
consists of the operators $A$ in $\mathscr {L}(\mathscr {H})$ such
that every $\{A\}^{\prime}$ contains a bounded maximal abelian set
of idempotents. We also found that for an operator $A$ in $\mathscr
{L}(\mathscr {H})$, the commutant $\{A\}^{\prime}$ may both contain
a bounded maximal abelian set of idempotents and an unbounded
maximal abelian set of idempotents.

Inspired by $(1.1)$, our second question is whether the equality
$(1.1)$ holds in the commutant $\{A\}^{\prime}$ for each operator
$A$ in $\mathscr{S}$. In \cite{Shi_2}, we gave a negative answer by
constructing an operator $C$ in the form
$$C=\begin{pmatrix}
N^{(\infty)}_{\mu}&I\\
0&N^{(\infty)}_{\mu}
\end{pmatrix}
\begin{matrix}
(L^2(\mu))^{(\infty)}\\
(L^2(\mu))^{(\infty)}\\
\end{matrix}. \eqno{(1.8)}$$

We denote by $\mathscr {S}^{}_{U}$ the subset of $\mathscr{S}$ such
that for every operator $A$ in $\mathscr {S}^{}_{U}$, the equality
$(1.1)$ holds for any two bounded maximal abelian sets of
idempotents in the commutant $\{A\}^{\prime}$. Compared with the
Jordan canonical form theorem, we define and say that {\textit{the
strongly irreducible decomposition of every operator $A$ in
$\mathscr {S}^{}_{U}$ is unique up to similarity}}. Therefore, our
third question is what the structure of an operator $A$ in $\mathscr
{S}^{}_{U}$ is. In \cite{Shi_2}, inspired by (\cite{Azoff_1},
Corollary $2$), the author mainly proved that an $n$-normal operator
$A$ in $\mathscr {L}(\mathscr {H})$ unitarily equivalent to the
following form is in $\mathscr {S}^{}_{U}$:
$$A=\begin{pmatrix}
N^{}_{\mu}&M^{}_{f^{}_{12}}&\cdots&M^{}_{f^{}_{1n}}\\
0&N^{}_{\mu}&\cdots&M^{}_{f^{}_{2n}}\\
\vdots&\vdots&\ddots&\vdots\\
0&0&\cdots&N^{}_{\mu}\\
\end{pmatrix}^{(m)}_{n\times n}, \eqno{(1.9)}$$
where $m,n<\infty$, $\mu$ and $N^{}_{\mu}$ are as in $(1.7)$,
$f^{}_{ij}$ is in $L^{\infty}(\mu)$ for $1\leq i,j\leq n$ and the
inequality $f^{}_{i,i+1}(\lambda)\neq 0$ holds for almost every
$\lambda$ in the support of $\mu$ and $1\leq i\leq n-1$.

In the present paper, we have two motivations. One is to generalize
the main result of \cite{Shi_2}. Precisely, we prove the operator
$$A=\bigoplus^{\infty}_{n=1}\bigoplus^{\infty}_{m=1} A^{(m)}_{nm} \eqno{(1.10)}$$ in
$\mathscr {L}(\mathscr {H})$ is in $\mathscr {S}^{}_{U}$, where
$A^{}_{nm}=0$ holds for all but finitely many $m$ and $n$ in
$\mathbb {N}$, $A_{nm}$ is unitarily equivalent to the form
$$\left(\begin{array}{cccccc}
N^{}_{\nu^{}_{nm}}&M^{}_{f^{}_{nm;12}}&\cdots&M^{}_{f^{}_{nm;1n}}\\
0&N^{}_{\nu^{}_{nm}}&\cdots&M^{}_{f^{}_{nm;2n}}\\
\vdots&\vdots&\ddots&\vdots\\
0&0&\cdots&N^{}_{\nu^{}_{nm}}\\
\end{array}\right)_{n\times n}, \eqno(1.11)$$
the measures $\nu^{}_{nm^{}_{1}}$ and $\nu^{}_{nm^{}_{2}}$ are
mutually singular compactly supported finite positive regular Borel
for $m^{}_{1}\neq m^{}_{2}$, the function $f^{}_{nm;ij}$ is in
$L^{\infty}(\nu^{}_{nm})$ for $1\leq i,j\leq n$ such that the
inequality
$$f^{}_{nm;i,i+1}(\lambda)\neq 0 \eqno{(1.12)}$$ holds for almost
every $\lambda$ in the support of $\nu^{}_{nm}$ for $1\leq i\leq
n-1$.

Since $\{A^{}_{nm^{}_{1}}\oplus A^{}_{nm^{}_{2}}\}^{\prime}=
\{A^{}_{nm^{}_{1}}\}^{\prime}\oplus\{A^{}_{nm^{}_{2}}\}^{\prime}$,
for the sake of simplicity and without loss of generality, the above
object can be fulfilled by proving that
$$A=A^{(m^{}_{1})}_{n^{}_{1}}\oplus A^{(m^{}_{2})}_{n^{}_{2}}\oplus
A^{(m^{}_{3})}_{n^{}_{3}} \eqno{(1.13)}$$ in $\mathscr {L}(\mathscr
{H})$ is in $\mathscr {S}^{}_{U}$, where $A^{}_{n^{}_{k}}$ is in the
form
$$A^{}_{n^{}_{k}}=\left(\begin{array}{cccccc}
N^{}_{\mu^{}_{}}&M^{}_{f^{}_{12,{k}}}&\cdots&M^{}_{f^{}_{1n^{}_{k},{k}}}\\
0&N^{}_{\mu^{}_{}}&\cdots&M^{}_{f^{}_{2n^{}_{k},{k}}}\\
\vdots&\vdots&\ddots&\vdots\\
0&0&\cdots&N^{}_{\mu^{}_{}}\\
\end{array}\right)_{n^{}_{k}\times n^{}_{k}}, \eqno(1.14)$$
for $k=1,2,3$, $n^{}_{1}> n^{}_{2}> n^{}_{3}$, the measure $\mu$ is
as in $(1.7)$, the function $f^{}_{ij,k}$ is in
$L^{\infty}(\mu^{}_{})$ for $1\leq i,j\leq n^{}_{k}$ such that the
inequality
$$f^{}_{i,i+1,k}(\lambda)\neq 0 \eqno{(1.15)}$$ holds for almost
every $\lambda$ in the support of $\mu$ for $1\leq i\leq
n^{}_{k}-1$, $1\leq k\leq 3$. The condition $(1.12)$ (or $(1.15)$)
is necessary and sufficient for an operator in the form of $(1.11)$
(or $(1.14)$) to be strongly irreducible almost everywhere on the
support of the corresponding measure in the sense of direct integral
decomposition as in $(1.5)$, which was proved in (\cite{Shi_2},
Lemma $3.1$).

The other motivation is to prove a complete similarity invariant of
an operator $A$ as in $(1.13)$ by $K$-theory for Banach algebras.
This similarity invariant is different from the necessary and
sufficient conditions for two $n$-normal operators similar to each
other proved by D. Deckard and C. Pearcy in (\cite{Deckard_1},
Theorem $1$).

Precisely, we prove the following theorems in this paper. In
$K$-theory for Banach algebras, by $V(\{A\}^{\prime}_{})$ we denote
the semigroup
$\cup^{\infty}_{n=1}\mathscr{P}^{}_{n}(\{A\}^{\prime})/\sim$, where
$\mathscr{P}^{}_{n}(\{A\}^{\prime})$ is the set of idempotents in
$M^{}_{n}(\{A\}^{\prime})$ and by ``$\sim$'' we denote that
similarity relation in the corresponding algebra. By
$K^{}_{0}(\{A\}^{\prime}_{})$ we denote the Grothendieck group
generated by $V(\{A\}^{\prime}_{})$, which is well known as the
$K^{}_{0}$-group of $\{A\}^{\prime}_{}$.

\begin{theorem}
Let $A\in\mathscr {L}(\mathscr {H})$ be assumed as in $(1.13)$. Then
the following statements hold:
\begin{enumerate}
\item [(a)] The strongly irreducible decomposition of $A$ is
unique up to similarity;
\item [(b)] $K^{}_{0}(\{A\}^{\prime})\cong\{f(\lambda)\in\mathbb{Z}^{(3)}:f
{\mbox{ is bounded and Borel on }}\sigma(A)\}.$
\end{enumerate}
\end{theorem}

When we deal with a finite direct sum of operators as in $(1.10)$,
inspired by (\cite{Conway_1}, Chapter $9$, Theorem $10.16$) we
obtain a generalization of the above theorem in the following form.

\begin{theorem}
Let $A\in\mathscr {L}(\mathscr {H})$ be assumed as in $(1.10)$. Then
the following statements hold:
\begin{enumerate}
\item [(a)] The strongly irreducible decomposition of $A$ is
unique up to similarity;
\item [(b)] There exists a bounded $\mathbb{N}$-valued simple
function $r^{}_{A}$ on $\sigma(A)$ such that
$$K^{}_{0}(\{A\}^{\prime})\cong
\{f(\lambda)\in\mathbb{Z}^{(r^{}_{A}(\lambda))}:f {\mbox{ is bounded
and Borel on }}\sigma(A)\}.$$
\end{enumerate}
\end{theorem}

For operators as in $(1.13)$, we characterize the similarity with
$K$-theory for Banach algebras as follows.

\begin{theorem}
Let $A=A^{(m^{}_{1})}_{n^{}_{1}}\oplus
A^{(m^{}_{2})}_{n^{}_{2}}\oplus A^{(m^{}_{3})}_{n^{}_{3}}$ and
$B=B^{(k^{}_{1})}_{l^{}_{1}}\oplus B^{(k^{}_{2})}_{l^{}_{2}}\oplus
B^{(k^{}_{3})}_{l^{}_{3}}$ be as in $(1.13)$ and every entry of $A$
and $B$ is in $L^{\infty}(\mu)$ as in $(1.14)$. Then $A$ and $B$ are
similar if and only if there exists a group isomorphism $\theta$
such that the following statements hold:
\begin{enumerate}
\item
$\theta(K^{}_{0}(\{A\oplus B\}^{\prime}))=
\{f(\lambda)\in\mathbb{Z}^{(3)}:f\ {\mbox{is\ bounded and Borel on
}}\sigma(A)\}$;
\item
$\theta([I^{}_{\{A\oplus B\}^{\prime}}])=
2m^{}_{1}e^{}_{1}+2m^{}_{2}e^{}_{2}+2m^{}_{3}e^{}_{3}$,
\end{enumerate}
where $\{e^{}_{i}(\lambda)\}^{3}_{i=1}$ are the generators of the
semigroup $\mathbb {N}^{(3)}$ of $\mathbb {Z}^{(3)}$ for every
$\lambda$ in $\sigma(A)$ and $I^{}_{\{A\oplus B\}^{\prime}}$ is the
unit of $\{A\oplus B\}^{\prime}$.
\end{theorem}

By a more complicated computation, we obtain a generalization of the
above theorem as follows.

\begin{theorem}
Let $A=\sum^{s}_{i=1}\oplus A^{(m^{}_{i})}_{n^{}_{i}}$ and
$B=\sum^{t}_{j=1}\oplus B^{(k^{}_{j})}_{l^{}_{j}}$ be in the sense
of $(1.13)$, and every entry of $A^{}_{n^{}_{i}}$ and
$B^{}_{l^{}_{j}}$ is in $L^{\infty}(\mu)$ as in $(1.14)$, for $1\leq
i\leq s<\infty$ and $1\leq j\leq t<\infty$ , where $n^{}_{i}\neq
n^{}_{j}$ for $i\neq j$, and $m^{}_{i}$, $n^{}_{i}$, $k^{}_{j}$ and
$l^{}_{j}$ are in $\mathbb {N}$ for every $i$ and $j$. Then $A$ and
$B$ are similar if and only if there exists a group isomorphism
$\theta$ such that the following statements hold:
\begin{enumerate}
\item
$\theta(K^{}_{0}(\{A\oplus B\}^{\prime}))=
\{f(\lambda)\in\mathbb{Z}^{(s)}:f\ {\mbox{is\ bounded and Borel on
}}\sigma(A)\}$;
\item
$\theta([I^{}_{\{A\oplus B\}^{\prime}}])=
2m^{}_{1}e^{}_{1}+2m^{}_{2}e^{}_{2}+\cdots+2m^{}_{s}e^{}_{s}$,
\end{enumerate}
where $\{e^{}_{i}(\lambda)\}^{s}_{i=1}$ are the generators of the
semigroup $\mathbb {N}^{(s)}$ of $\mathbb {Z}^{(s)}$ for every
$\lambda$ in $\sigma(A)$ and $I^{}_{\{A\oplus B\}^{\prime}}$ is the
unit of $\{A\oplus B\}^{\prime}$.
\end{theorem}

Let the support of every spectral measure $\nu^{}_{nm}$ in the sense
of $(1.10)$ and $(1.11)$ be a single point in $\mathbb{C}$, then
Theorem $1.3$ shows that the strongly irreducible decomposition of
every matrix $A$ in $M^{}_{n}(\mathbb{C})$ is unique up to
similarity, and Theorem $1.4$ characterizes a necessary and
sufficient condition that two matrices are similar. This is
identified with the Jordan canonical form theorem.

This paper is organized as follows. In section $2$, we prove Theorem
$1.2$ and Theorem $1.4$. In section $3$, we develop a method of
decomposing an upper triangular $n$-normal operator $A$ of the
following form with respect to the multiplicity function of the
$(1,1)$ entry:
$$A=\begin{pmatrix}
M^{}_{f^{}_{}}&M^{}_{f^{}_{12}}&\cdots&M^{}_{f^{}_{1n}}\\
0&M^{}_{f^{}_{}}&\cdots&M^{}_{f^{}_{2n}}\\
\vdots&\vdots&\ddots&\vdots\\
0&0&\cdots&M^{}_{f^{}_{}}\\
\end{pmatrix}^{}_{n\times n}
\begin{matrix}
L^{2}(\mu)\\
L^{2}(\mu)\\
\vdots\\
L^{2}(\mu)\\
\end{matrix}, \eqno{(1.16)}$$
where $n<\infty$, $\mu$ are as in $(1.7)$, $f$ and $f^{}_{ij}$ are
in $L^{\infty}(\mu)$ for $1\leq i,j\leq n$ and the inequality
$f^{}_{i,i+1}(\lambda)\neq 0$ holds for almost every $\lambda$ in
the support of $\mu$ and $1\leq i\leq n-1$.

\section{Proofs}

For an $n$-normal operator $A^{}_{}$ in the form as in $(1.6)$, an
application of (\cite{Shi_2}, Lemma $3.1$) shows that for a fixed
$\lambda$ in the support of $\nu$, the operator $A^{}_{}(\lambda)$
is strongly irreducible if and only if
$f^{}_{ii}(\lambda)=f^{}_{nn}(\lambda)$ and
$f^{}_{i,i+1}(\lambda)\neq 0$ hold for $(1\leq i\leq n-1)$.
Therefore for an $n$-normal operator $A$ in the form as in $(1.14)$
and $(1.15)$, $A(\lambda)$ is strongly irreducible for almost every
$\lambda$ in the support of $\mu$ in the sense of $(1.5)$. We need
to mention that the multiplication operators $M^{}_{f^{}_{i,i+1,k}}$
may not be invertible in general. This makes the computation become
more complicated. However, the commutant $\{A^{}_{n}\}^{\prime}_{}$
is a subalgebra of $\{N^{(n)}_{\mu^{}_{}}\}^{\prime}$ by
(\cite{Shi_2}, Lemma $3.2$) for an operator $A^{}_{n}$ in the form
$$A^{}_{n}=\left(\begin{array}{cccccc}
N^{}_{\mu^{}_{}}&M^{}_{f^{}_{12}}&\cdots&M^{}_{f^{}_{1n^{}_{}}}\\
0&N^{}_{\mu^{}_{}}&\cdots&M^{}_{f^{}_{2n^{}_{}}}\\
\vdots&\vdots&\ddots&\vdots\\
0&0&\cdots&N^{}_{\mu^{}_{}}\\
\end{array}\right)_{n^{}\times n^{}}, \eqno(2.1)$$
where the measure $\mu$ is as in $(1.7)$, the function $f^{}_{ij}$
is in $L^{\infty}(\mu^{}_{})$ for $1\leq i,j\leq n$ such that the
inequality
$$f^{}_{i,i+1}(\lambda)\neq 0 \eqno{(2.2)}$$ holds for almost
every $\lambda$ in the support of $\mu$ for $1\leq i\leq n^{}_{}-1$.
Precisely, by (\cite{Shi_2}, Lemma $3.2$), every operator $X^{}_{}$
in $\{A^{}_{n}\}^{\prime}_{}$ is in the form
$$X^{}_{}=\left(\begin{array}{cccccc}
M^{}_{\psi^{}_{}}&M^{}_{\psi^{}_{12}}&M^{}_{\psi^{}_{13}}&\cdots&M^{}_{\psi^{}_{1n}}\\
0&M^{}_{\psi^{}_{}}&M^{}_{\psi^{}_{23}}&\cdots&M^{}_{\psi^{}_{2n}}\\
0&0&M^{}_{\psi^{}_{}}&\cdots&M^{}_{\psi^{}_{3n}}\\
\vdots&\vdots&\vdots&\ddots&\vdots\\
0&0&0&\cdots&M^{}_{\psi^{}_{}}\\
\end{array}\right)^{}_{n\times n}, \eqno(2.3)$$
and in special, every idempotent $E$ in $\{A^{}_{n}\}^{\prime}$ is
in the form $E=M^{(n)}_{\chi^{}_{\Delta}}$ for some characteristic
function $\chi^{}_{\Delta}$ in $L^{\infty}(\mu^{}_{})$, where
$\Delta$ is a Borel subset in the support of $\mu^{}_{}$. Let
$\mathscr{E}^{}_{n}$ denote the set of idempotents in
$\{A^{}_{n}\}^{\prime}_{}$. Then $\mathscr{E}^{}_{n}$ is the only
maximal abelian set of idempotents in $\{A^{}_{n}\}^{\prime}_{}$ and
obviously, the set $\mathscr{E}^{}_{n}$ is bounded. We observe that
the bounded set of idempotents $$\mathscr{E}^{}_{}\triangleq
(\mathscr{E}^{}_{n^{}_{1}}\oplus\cdots\oplus\mathscr{E}^{}_{n^{}_{1}})
\oplus(\mathscr{E}^{}_{n^{}_{2}}\oplus\cdots\oplus\mathscr{E}^{}_{n^{}_{2}})
\oplus(\mathscr{E}^{}_{n^{}_{3}}\oplus\cdots\oplus\mathscr{E}^{}_{n^{}_{3}})
\eqno{(2.4)}$$ ($m^{}_{1}$ copies of $\mathscr{E}^{}_{n^{}_{1}}$,
$m^{}_{2}$ copies of $\mathscr{E}^{}_{n^{}_{2}}$, and $m^{}_{3}$
copies of $\mathscr{E}^{}_{n^{}_{3}}$) is maximal abelian in the
commutant of $A=A^{(m^{}_{1})}_{n^{}_{1}}\oplus
A^{(m^{}_{2})}_{n^{}_{2}}\oplus A^{(m^{}_{3})}_{n^{}_{3}}$ as
mentioned from $(1.13)$ to $(1.15)$. In the rest of this article, we
define $\mathscr{E}^{}_{}$ to be the \textit{standard} bounded
maximal abelian set of idempotents in $\{A\}^{\prime}_{}$ where $A$
is defined as in $(1.13)$. The following two preliminary lemmas are
needed to prove Theorem $1.2$.

\begin{lemma}
Let $A^{}_{n^{}_{1}}$ and $A^{}_{n^{}_{2}}$ $(n^{}_{1}>n^{}_{2})$ be
assumed as in $(1.14)$. Then the following statements hold:
\begin{enumerate}
\item the equality $A^{}_{n^{}_{1}}X=XA^{}_{n^{}_{2}}$ yields that
$X=(X^{\textup{T}}_{1},{\mathbf
0}_{n^{}_{2}\times(n^{}_{1}-n^{}_{2})})^{\textup{T}}$, where
$X^{}_{1}$ is an upper triangular $n^{}_{2}$-by-$n^{}_{2}$
operator-valued matrix such that every entry of $X^{}_{1}$ is in
$\{N^{}_{\mu^{}_{}}\}^{\prime}$, and the transpose of $X^{}_{1}$ is
denoted by $X^{\textup{T}}_{1}$;
\item the equality $A^{}_{n^{}_{2}}Y=YA^{}_{n^{}_{1}}$ yields that
$Y=({\mathbf 0}_{n^{}_{2}\times(n^{}_{1}-n^{}_{2})},Y^{}_{1})$,
where $Y^{}_{1}$ is an upper triangular $n^{}_{2}$-by-$n^{}_{2}$
operator-valued matrix such that every entry of $Y^{}_{1}$ is in
$\{N^{}_{\mu^{}_{}}\}^{\prime}_{}$.
\end{enumerate}
\end{lemma}

\begin{proof}
If $A^{}_{n^{}_{1}}=A^{}_{n^{}_{2}}$, then this lemma is identified
with (\cite{Shi_2}, Lemma $3.2$). For the sake of simplicity, let
operators $A^{}_{n^{}_{1}}$ and $A^{}_{n^{}_{2}}$ be in the form
$$A^{}_{n^{}_{2}}=\left(\begin{array}{cccccc}
N^{}_{\mu^{}_{}}&M^{}_{f^{}_{12}}&M^{}_{f^{}_{13}}&\cdots&M^{}_{f^{}_{1n^{}_{2}}}\\
0&N^{}_{\mu^{}_{}}&M^{}_{f^{}_{23}}&\cdots&M^{}_{f^{}_{2n^{}_{2}}}\\
0&0&N^{}_{\mu^{}_{}}&\cdots&M^{}_{f^{}_{3n^{}_{2}}}\\
\vdots&\vdots&\vdots&\ddots&\vdots\\
0&0&0&\cdots&N^{}_{\mu^{}_{}}\\
\end{array}\right)_{n^{}_{2}\times n^{}_{2}}
\begin{matrix}
L^{2}(\mu)\\
L^{2}(\mu)\\
L^{2}(\mu)\\
\vdots\\
L^{2}(\mu)\\
\end{matrix} \eqno(2.5)$$ and
$$A^{}_{n^{}_{1}}=\left(\begin{array}{cccccc}
N^{}_{\mu^{}_{}}&M^{}_{g^{}_{12}}&M^{}_{g^{}_{13}}&\cdots&M^{}_{g^{}_{1n^{}_{1}}}\\
0&N^{}_{\mu^{}_{}}&M^{}_{g^{}_{23}}&\cdots&M^{}_{g^{}_{2n^{}_{1}}}\\
0&0&N^{}_{\mu^{}_{}}&\cdots&M^{}_{g^{}_{3n^{}_{1}}}\\
\vdots&\vdots&\vdots&\ddots&\vdots\\
0&0&0&\cdots&N^{}_{\mu^{}_{}}\\
\end{array}\right)_{n^{}_{1}\times n^{}_{1}}
\begin{matrix}
L^{2}(\mu)\\
L^{2}(\mu)\\
L^{2}(\mu)\\
\vdots\\
L^{2}(\mu)\\
\end{matrix}. \eqno(2.6)$$

Let $E^{}_{\mu}(\cdot)$ be the spectral measure for $N^{}_{\mu}$.
For a Borel subset $\Delta$ of $\sigma(N^{}_{\mu})$ such that
$E^{}_{\mu}(\Delta)$ is a nontrivial projection in
$\{N^{}_{\mu}\}^{\prime}$, we write $P^{}_{1}=E^{}_{\mu}(\Delta)$
and $P^{}_{2}=E^{}_{\mu}(\sigma(N^{}_{\mu})\backslash\Delta)$,
meanwhile write $\mu^{}_{1}$ for $\mu|^{}_{\Delta}$ and $\mu^{}_{2}$
for $\mu|^{}_{\sigma(N^{}_{\mu})\backslash\Delta}$. Hence the
operators $A^{}_{n^{}_{1}}$, $A^{}_{n^{}_{2}}$ and $X$ can be
expressed in the form
$$A^{}_{n^{}_{1}}=\begin{pmatrix}
A^{}_{n^{}_{1},1}&\bf{0}\\
\bf{0}&A^{}_{n^{}_{1},2}
\end{pmatrix}
\begin{matrix}
\mbox{ran}P^{(n^{}_{1})}_{1}\\
\mbox{ran}P^{(n^{}_{1})}_{2}\\
\end{matrix},\quad
A^{}_{n^{}_{2}}=\begin{pmatrix}
A^{}_{n^{}_{2},1}&\bf{0}\\
\bf{0}&A^{}_{n^{}_{2},2}
\end{pmatrix}
\begin{matrix}
\mbox{ran}P^{(n^{}_{2})}_{1}\\
\mbox{ran}P^{(n^{}_{2})}_{2}\\
\end{matrix},   $$
$$\mbox{and}\quad
X=\begin{pmatrix}
X^{}_{11}&X^{}_{12}\\
X^{}_{21}&X^{}_{22}\\
\end{pmatrix} \eqno(2.7)$$ where
$$A^{}_{n^{}_{1},i}=\left(\begin{array}{cccccc}
N^{}_{\mu^{}_{i}}&M^{}_{g^{}_{12,i}}&M^{}_{g^{}_{13,i}}&\cdots&M^{}_{g^{}_{1n^{}_{1},i}}\\
0&N^{}_{\mu^{}_{i}}&M^{}_{g^{}_{23,i}}&\cdots&M^{}_{g^{}_{2n^{}_{1},i}}\\
0&0&N^{}_{\mu^{}_{i}}&\cdots&M^{}_{g^{}_{3n^{}_{1},i}}\\
\vdots&\vdots&\vdots&\ddots&\vdots\\
0&0&0&\cdots&N^{}_{\mu^{}_{i}}\\
\end{array}\right)_{n^{}_{1}\times n^{}_{1}}
\begin{matrix}
\mbox{ran}P^{}_{i}\\
\mbox{ran}P^{}_{i}\\
\mbox{ran}P^{}_{i}\\
\vdots\\
\mbox{ran}P^{}_{i}\\
\end{matrix},\quad i=1,2, \eqno(2.8)$$
and
$$A^{}_{n^{}_{2},i}=\left(\begin{array}{cccccc}
N^{}_{\mu^{}_{i}}&M^{}_{f^{}_{12,i}}&M^{}_{f^{}_{13,i}}&\cdots&M^{}_{f^{}_{1n^{}_{2},i}}\\
0&N^{}_{\mu^{}_{i}}&M^{}_{f^{}_{23,i}}&\cdots&M^{}_{f^{}_{2n^{}_{2},i}}\\
0&0&N^{}_{\mu^{}_{i}}&\cdots&M^{}_{f^{}_{3n^{}_{2},i}}\\
\vdots&\vdots&\vdots&\ddots&\vdots\\
0&0&0&\cdots&N^{}_{\mu^{}_{i}}\\
\end{array}\right)_{n^{}_{2}\times n^{}_{2}}
\begin{matrix}
\mbox{ran}P^{}_{i}\\
\mbox{ran}P^{}_{i}\\
\mbox{ran}P^{}_{i}\\
\vdots\\
\mbox{ran}P^{}_{i}\\
\end{matrix},\quad i=1,2. \eqno(2.9)$$
The equality $A^{}_{n^{}_{1}}X=XA^{}_{n^{}_{2}}$ yields
$A^{}_{n^{}_{1},1}X^{}_{12}=X^{}_{12}A^{}_{n^{}_{2},2}$. And this
equality can be expressed in the form
$$\begin{pmatrix}
N^{}_{\mu^{}_{1}}&M^{}_{g^{}_{12,1}}&\cdots&M^{}_{g^{}_{1n^{}_{1},1}}\\
0&N^{}_{\mu^{}_{1}}&\cdots&M^{}_{g^{}_{2n^{}_{1},1}}\\
\vdots&\vdots&\ddots&\vdots\\
0&0&\cdots&N^{}_{\mu^{}_{1}}\\
\end{pmatrix}
\begin{pmatrix}
X^{}_{12,11}&X^{}_{12,12}&\cdots&X^{}_{12,1n^{}_{2}}\\
X^{}_{12,21}&X^{}_{12,22}&\cdots&X^{}_{12,2n^{}_{2}}\\
\vdots&\vdots&\ddots&\vdots\\
X^{}_{12,n^{}_{1}1}&X^{}_{12,n^{}_{1}2}&\cdots&X^{}_{12,n^{}_{1}n^{}_{2}}\\
\end{pmatrix} \eqno(2.10)$$
$$=\begin{pmatrix}
X^{}_{12,11}&X^{}_{12,12}&\cdots&X^{}_{12,1n^{}_{2}}\\
X^{}_{12,21}&X^{}_{12,22}&\cdots&X^{}_{12,2n^{}_{2}}\\
\vdots&\vdots&\ddots&\vdots\\
X^{}_{12,n^{}_{1}1}&X^{}_{12,n^{}_{1}2}&\cdots&X^{}_{12,n^{}_{1}n^{}_{2}}\\
\end{pmatrix}
\begin{pmatrix}
N^{}_{\mu^{}_{2}}&M^{}_{f^{}_{12,2}}&\cdots&M^{}_{f^{}_{1n^{}_{2},2}}\\
0&N^{}_{\mu^{}_{2}}&\cdots&M^{}_{f^{}_{2n^{}_{2},2}}\\
\vdots&\vdots&\ddots&\vdots\\
0&0&\cdots&N^{}_{\mu^{}_{2}}\\
\end{pmatrix}.$$

Since the measures $\mu^{}_{1}$ and $\mu^{}_{2}$ are mutually
singular, the equality
$N^{}_{\mu^{}_{1}}X^{}_{12,n^{}_{1}1}=X^{}_{12,n^{}_{1}1}N^{}_{\mu^{}_{2}}$
yields that $X^{}_{12,n^{}_{1}1}=0$. Thus the equality
$N^{}_{\mu^{}_{1}}X^{}_{12,n^{}_{1}2}=X^{}_{12,n^{}_{1}2}N^{}_{\mu^{}_{2}}$
yields that $X^{}_{12,n^{}_{1}2}=0$. By this method, we obtain that
every entry in the $n^{}_{1}$-th row of $X^{}_{12}$ is zero. The
same result holds for the the $(n^{}_{1}-1)$-th row of $X^{}_{12}$.
By induction, we obtain that $X^{}_{12}=\bf{0}$. By a similar
discussion, we have that $X^{}_{21}=\bf{0}$. This means that the
equality $P^{(n^{}_{1})}_{i}X=XP^{(n^{}_{2})}_{i}$ holds for every
Borel subset $\Delta$ of $\sigma(N^{}_{\mu})$. Therefore the
operator $X$ can be expressed in the form
$$X=\begin{pmatrix}
M^{}_{h^{}_{11}}&M^{}_{h^{}_{12}}&M^{}_{h^{}_{13}}&\cdots&M^{}_{h^{}_{1n^{}_{2}}}\\
M^{}_{h^{}_{21}}&M^{}_{h^{}_{22}}&M^{}_{h^{}_{23}}&\cdots&M^{}_{h^{}_{2n^{}_{2}}}\\
M^{}_{h^{}_{31}}&M^{}_{h^{}_{32}}&M^{}_{h^{}_{33}}&\cdots&M^{}_{h^{}_{3n^{}_{2}}}\\
\vdots&\vdots&\vdots&\ddots&\vdots\\
M^{}_{h^{}_{n^{}_{1}1}}&M^{}_{h^{}_{n^{}_{1}2}}&M^{}_{h^{}_{n^{}_{1}3}}&\cdots&M^{}_{h^{}_{n^{}_{1}n^{}_{2}}}\\
\end{pmatrix}_{n^{}_{1}\times n^{}_{2}}, \eqno(2.11)$$
where $h^{}_{ij}$ is in $L^{\infty}(\mu)$, $1\leq i\leq n^{}_{1}$
and $1\leq j\leq n^{}_{2}$. By the assumption, we have that
$f^{}_{i,i+1}(\lambda)\neq 0$ and $g^{}_{j,j+1}(\lambda)\neq 0$ for
$1\leq i\leq n^{}_{2}-1$, $1\leq j\leq n^{}_{1}-1$, and almost every
$\lambda$ in $\sigma(N^{}_{\mu})$. The equality
$A^{}_{n^{}_{1}}X=XA^{}_{n^{}_{2}}$ yields that
$$N^{}_{\mu^{}_{}}M^{}_{h^{}_{n^{}_{1}-1,1}}
+M^{}_{g^{}_{n^{}_{1}-1,n^{}_{1}}}M^{}_{h^{}_{n^{}_{1}1}}
=M^{}_{h^{}_{n^{}_{1}-1,1}}N^{}_{\mu^{}_{}}. \eqno{(2.12)}$$ This
equality yields that $M^{}_{h^{}_{n^{}_{1}1}}=0$. Thus the equality
$$N^{}_{\mu^{}_{}}M^{}_{h^{}_{n^{}_{1}-2,1}}
+M^{}_{g^{}_{n^{}_{1}-2,n^{}_{1}-1}}M^{}_{h^{}_{n^{}_{1}-1,1}}
=M^{}_{h^{}_{n^{}_{1}-2,1}}N^{}_{\mu^{}_{}} \eqno{(2.13)}$$ yields
that $M^{}_{h^{}_{n^{}_{1}-1,1}}=0$. By computation, we obtain that
$M^{}_{h^{}_{j,1}}=0$ for $2\leq j\leq n^{}_{1}$.

By the equality $A^{}_{n^{}_{1}}X=XA^{}_{n^{}_{2}}$, we have
$$N^{}_{\mu^{}_{}}M^{}_{h^{}_{n^{}_{1}-1,2}}
+M^{}_{g^{}_{n^{}_{1}-1,n^{}_{1}}}M^{}_{h^{}_{n^{}_{1}2}}
=M^{}_{h^{}_{n^{}_{1}-1,2}}N^{}_{\mu^{}_{}}. \eqno{(2.14)}$$ This
yields that $M^{}_{h^{}_{n^{}_{1}2}}=0$. Thus the equality
$$N^{}_{\mu^{}_{}}M^{}_{h^{}_{n^{}_{1}-2,2}}
+M^{}_{g^{}_{n^{}_{1}-2,n^{}_{1}-1}}M^{}_{h^{}_{n^{}_{1}-1,2}}
=M^{}_{h^{}_{n^{}_{1}-2,2}}N^{}_{\mu^{}_{}} \eqno{(2.15)}$$ yields
that $M^{}_{h^{}_{n^{}_{1}-1,2}}=0$. By computation, we obtain that
$M^{}_{h^{}_{j,2}}=0$ for $3\leq j\leq n^{}_{1}$. By induction, we
have $M^{}_{h^{}_{j,i}}=0$ for $i<j$. The proof of the first
assertion is finished.

In the proof of the second assertion, by a similar computation, we
obtain that $Y$ is an $n^{}_{2}$-by-$n^{}_{1}$ operator-valued
matrix as in $(2.11)$. Therefore, we apply the equality
$A^{}_{n^{}_{2}}Y=YA^{}_{n^{}_{1}}$ to obtain that
$$N^{}_{\mu^{}_{}}M^{}_{h^{}_{n^{}_{2}2}}
=M^{}_{h^{}_{n^{}_{2}1}}M^{}_{g^{}_{12}}
+M^{}_{h^{}_{n^{}_{2}2}}N^{}_{\mu^{}_{}}. \eqno{(2.16)}$$ This
equality yields that $M^{}_{h^{}_{n^{}_{2}1}}=0$. Thus the equality
$$N^{}_{\mu^{}_{}}M^{}_{h^{}_{n^{}_{2}3}}
=M^{}_{h^{}_{n^{}_{2}2}}M^{}_{g^{}_{23}}
+M^{}_{h^{}_{n^{}_{2}3}}N^{}_{\mu^{}_{}} \eqno{(2.17)}$$ yields that
$M^{}_{h^{}_{n^{}_{2}2}}=0$. By computation, we obtain that
$M^{}_{h^{}_{n^{}_{2}j}}=0$ for $1\leq j\leq n^{}_{1}-1$.

By the equality $A^{}_{n^{}_{2}}Y=YA^{}_{n^{}_{1}}$, we have
$$N^{}_{\mu^{}_{}}M^{}_{h^{}_{n^{}_{2}-1,2}}
=M^{}_{h^{}_{n^{}_{2}-1,1}}M^{}_{g^{}_{12}}
+M^{}_{h^{}_{n^{}_{2}-1,2}}N^{}_{\mu^{}_{}}. \eqno{(2.18)}$$ This
yields that $M^{}_{h^{}_{n^{}_{2}-1,1}}=0$. Thus the equality
$$N^{}_{\mu^{}_{}}M^{}_{h^{}_{n^{}_{2}-1,3}}
=M^{}_{h^{}_{n^{}_{2}-1,2}}M^{}_{g^{}_{23}}
+M^{}_{h^{}_{n^{}_{2}-1,3}}N^{}_{\mu^{}_{}} \eqno{(2.19)}$$ yields
that $M^{}_{h^{}_{n^{}_{2}-1,2}}=0$. By computation, we obtain that
$M^{}_{h^{}_{n^{}_{2}-1,j}}=0$ for $1\leq j\leq n^{}_{1}-2$. By
induction, we have $M^{}_{h^{}_{i,j}}=0$ for $j\leq
n^{}_{1}-n^{}_{2}+i-1$. Therefore, the proof of the second assertion
is finished.

A fact we need to mention is that if $n^{}_{1}=n^{}_{2}$, then $X$
is an $n^{}_{1}$-by-$n^{}_{1}$ upper triangular operator-valued
matrix such that every entry of $X$ is in
$\{N^{}_{\mu^{}_{}}\}^{\prime}$ and the entries of $X$ have further
relations with others.
\end{proof}

\begin{lemma}
For an operator $A$ defined from $(1.13)$ to $(1.15)$ and every
idempotent $P$ in $\{A\}^{\prime}$, there exists an invertible
operator $X$ in $\{A\}^{\prime}$ such that $XPX^{-1}_{}$ is in
$\mathscr {E}$ $($defined as in $(2.4)$$)$.
\end{lemma}

\begin{proof}
As defined from $(1.13)$ to $(1.15)$, we have
$A=A^{(m^{}_{1})}_{n^{}_{1}}\oplus A^{(m^{}_{2})}_{n^{}_{2}}\oplus
A^{(m^{}_{3})}_{n^{}_{3}}$ for positive integers
$n^{}_{1}>n^{}_{2}>n^{}_{3}$.

Let $B$ be an operator in $\{A\}^{\prime}$. Then $B$ can be
expressed in the form
$$B=\begin{pmatrix}
B^{}_{11}&B^{}_{12}&B^{}_{13}&\\
B^{}_{21}&B^{}_{22}&B^{}_{23}&\\
B^{}_{31}&B^{}_{32}&B^{}_{33}&\\
\end{pmatrix},\eqno(2.20)$$
where
$$B^{}_{ij}=\begin{pmatrix}
B^{}_{{ij};11}&\cdots&B^{}_{ij;1m^{}_{j}}\\
\vdots&\ddots&\vdots\\
B^{}_{ij;m^{}_{i}1}&\cdots&B^{}_{ij;m^{}_{i}m^{}_{j}}\\
\end{pmatrix}^{}_{m^{}_{i}\times m^{}_{j}},\eqno(2.21)$$
and $B^{}_{ij;st}$ is in the set $\{X\mbox{ is bounded
linear}:A^{}_{n^{}_{i}}X=XA^{}_{n^{}_{j}}\}$, for $1\leq i,j\leq 3$.
For $B$ in $\{A\}^{\prime}$, there exists a unitary operator $U$
which is a composition of finitely many row-switching
transformations such that $C=UBU^{*}$ is in the form
$$C=\begin{pmatrix}
C^{}_{11}&\cdots&C^{}_{1n^{}_{1}}\\
\vdots&\ddots&\vdots\\
C^{}_{n^{}_{1}1}&\cdots&C^{}_{n^{}_{1}n^{}_{1}}\\
\end{pmatrix},\eqno(2.22)$$
where $C^{}_{lk}$ consists of the $(l,k)$ entries of each
$B^{}_{ij;st}$, and the relative positions of these entries stay
invariant in $C^{}_{lk}$. Notice that $C^{}_{lk}$ is not square for
$l\neq k$, and $C^{}_{11}$, $C^{}_{n^{}_{3}+1,n^{}_{3}+1}$ and
$C^{}_{n^{}_{2}+1,n^{}_{2}+1}$ are not of the same size. By Lemma
$2.1$, we have that $C^{}_{ij}=0$ for $i>j$.

For $1\leq i\leq n_{3}$, the block entry $C^{}_{ii}$ is in the form
$$C^{}_{ii}=\begin{pmatrix}
C^{}_{ii;11}&C^{}_{ii;12}&C^{}_{ii;13}\\
\bf{0}&C^{}_{ii;22}&C^{}_{ii;23}\\
\bf{0}&\bf{0}&C^{}_{ii;33}\\
\end{pmatrix},\eqno(2.23)$$
where
$$C^{}_{ii;kl}=\begin{pmatrix}
b^{ii}_{kl;11}&\cdots&b^{ii}_{kl;1m^{}_{l}}\\
\vdots&\ddots&\vdots\\
b^{ii}_{kl;m^{}_{k}1}&\cdots&b^{ii}_{kl;m^{}_{k}m^{}_{l}}\\
\end{pmatrix}^{}_{m^{}_{k}\times m^{}_{l}},\eqno(2.24)$$
and the operator $b^{ii}_{kl;st}$ is the $(i,i)$ entry of the block
$B^{}_{kl;st}$, for $1\leq k,l\leq 3$, and $1\leq s\leq m^{}_{k}$,
and $1\leq t\leq m^{}_{l}$.

For $n^{}_{3}< i\leq n^{}_{2}$, the block entry $C^{}_{ii}$ is in
the form
$$\begin{pmatrix}
b^{ii}_{11;11}&\cdots&b^{ii}_{11;1m^{}_{1}}
&b^{ii}_{12;11}&\cdots&b^{ii}_{12;1m^{}_{2}}\\
\vdots&\ddots&\vdots&\vdots&\ddots&\vdots\\
b^{ii}_{11;m^{}_{1}1}&\cdots&b^{ii}_{11;m^{}_{1}m^{}_{1}}
&b^{ii}_{12;m^{}_{1}1}&\cdots&b^{ii}_{12;m^{}_{1}m^{}_{2}}\\
&&&b^{ii}_{22;11}&\cdots&b^{ii}_{22;1m^{}_{2}}\\
&{\bf{0}}_{m^{}_{2}\times m^{}_{1}}&&\vdots&\ddots&\vdots\\
&&&b^{ii}_{22;m^{}_{2}1}&\cdots&b^{ii}_{22;m^{}_{2}m^{}_{2}}\\
\end{pmatrix}, \eqno(2.25)$$

and for $n^{}_{2}< j\leq n^{}_{1}$ the block entry $C^{}_{jj}$ is in
the form
$$\begin{pmatrix}
b^{jj}_{11;11}&\cdots&b^{jj}_{11;1m^{}_{1}}\\
\vdots&\ddots&\vdots\\
b^{jj}_{11;m^{}_{1}1}&\cdots&b^{jj}_{11;m^{}_{1}m^{}_{1}}\\
\end{pmatrix},\eqno(2.26)$$
where the operator $b^{ii}_{kl;st}$ is the $(i,i)$ entry of the
block $B^{}_{kl;st}$, for $1\leq k,l\leq 2$, and $1\leq s\leq
m^{}_{k}$, and $1\leq t\leq m^{}_{l}$, and the operator
$b^{jj}_{11;st}$ is the $(j,j)$ entry of the block $B^{}_{11;st}$,
for $1\leq s\leq m^{}_{1}$, and $1\leq t\leq m^{}_{1}$.

Let $C^{\prime}_{ii}$ be the block diagonal matrix in which the
diagonal blocks are the same as in $C^{}_{ii}$. For example, the
operator $C^{\prime}_{11}$ is in the form
$$C^{\prime}_{11}=\begin{pmatrix}
C^{}_{11;11}&\bf{0}&\bf{0}\\
\bf{0}&C^{}_{11;22}&\bf{0}\\
\bf{0}&\bf{0}&C^{}_{11;33}\\
\end{pmatrix}. \eqno(2.27)$$

We observe that an operator $C^{\prime}$ in the form
$$C^{\prime}=\begin{pmatrix}
C^{\prime}_{11}&\bf{0}&\cdots&\bf{0}\\
\bf{0}&C^{\prime}_{22}&\cdots&\bf{0}\\
\vdots&\vdots&\ddots&\vdots\\
\bf{0}&\bf{0}&\cdots&C^{\prime}_{n^{}_{1}n^{}_{1}}\\
\end{pmatrix} \eqno(2.28)$$
is in the commutant $\{UAU^{*}\}^{\prime}$. Let
$\sigma^{}_{{\{UAU^{*}\}^{\prime}}}(C-C^{\prime})$ denote the
spectrum of $C-C^{\prime}$ in the unital Banach algebra
$\{UAU^{*}\}^{\prime}$. Then for every operator $D$ in the commutant
$\{UAU^{*}\}^{\prime}$, we obtain the following equality
$$\sigma^{}_{{\{UAU^{*}\}^{\prime}}}(D(C-C^{\prime}))
=\sigma^{}_{{\{UAU^{*}\}^{\prime}}}((C-C^{\prime})D)=\{0\}.
\eqno(2.29)$$ Therefore, the operator $C-C^{\prime}$ is in the
Jacobson radical of $\{UAU^{*}\}^{\prime}$ denoted by
$\mbox{Rad}(\{UAU^{*}\}^{\prime})$.

Let $C$ be an idempotent in $\{UAU^{*}\}^{\prime}$. Then
$C^{\prime}$ is also an idempotent in $\{UAU^{*}\}^{\prime}$. Notice
that $2C^{\prime}-I$ is invertible in $\{UAU^{*}\}^{\prime}$. Then
the equality
$$(2C^{\prime}-I)(C+C^{\prime}-I)=I+(2C^{\prime}-I)(C-C^{\prime})
\eqno(2.30)$$ yields that the operator $C+C^{\prime}-I$ is
invertible in $\{UAU^{*}\}^{\prime}$, since $C-C^{\prime}$ is in
$\mbox{Rad}(\{UAU^{*}\}^{\prime})$. Therefore, we obtain the
equality $(C+C^{\prime}-I)C=C^{\prime}(C+C^{\prime}-I)$ which means
that the operators $C$ and $C^{\prime}$ are similar in
$\{UAU^{*}\}^{\prime}$.

Next, it suffices to show that the $(1,1)$ block of
$C^{\prime}_{11}$ denoted by $C^{}_{11;11}$ is similar to an element
of the standard bounded maximal abelian set of idempotents in
$M^{}_{m^{}_{1}}(L^{\infty}(\mu^{}_{}))$.

We assert that for every positive integer $k$, there exists a
positive integer $l^{}_{k}$ such that for every idempotent $P$ in
$\mathscr{L}(\mathscr{H})$ satisfying $\|P\|\leq k$, there exists an
invertible operator $X$ in $\mathscr{L}(\mathscr{H})$ satisfying
$\|X\|\leq l^{}_{k}$ and $\|X^{-1}\|\leq l^{}_{k}$ such that
$XPX^{-1}$ is the corresponding Jordan canonical form. The idea is
from considering the the following equality
$$\begin{pmatrix}
I&R\\
0&I\\
\end{pmatrix}
\begin{pmatrix}
I&R\\
0&0\\
\end{pmatrix}
\begin{pmatrix}
I&-R\\
0&I\\
\end{pmatrix}=
\begin{pmatrix}
I&0\\
0&0\\
\end{pmatrix},\quad\mbox{for}\quad
P=\begin{pmatrix}
I&R\\
0&0\\
\end{pmatrix}.\eqno(2.31)$$
Therefore, for a set as in (\cite{Azoff_1}, Corollary $3$)
$$\begin{array}{r}
\mathscr{S}^{}_{l^{}_{k}}=\{(S,J,Y)\in M^{}_{m}(\mathbb{C})\times
M^{}_{m}(\mathbb{C})\times
M^{}_{m}(\mathbb{C}): J\ {\mbox{is in Jordan form}},\quad\\
\|Y\|\leq l^{}_{k},\ \|Y^{-1}\|\leq l^{}_{k}{\mbox{ and }}
YSY^{-1}=J\},
\end{array}\eqno(2.32)$$
the set $\pi^{}_{1}(\mathscr{S}^{}_{l^{}_{k}})$ contains every
idempotent with norm less than $k$. By (\cite{Azoff_1}, Theorem
$1$), we obtain that the Borel map
$\phi^{}_{l^{}_{k}}:\pi^{}_{1}(\mathscr{S}^{}_{l^{}_{k}})\rightarrow\pi^{}_{3}(\mathscr{S}^{}_{l^{}_{k}})$
is bounded. Therefore the equivalent class of
$$\phi^{}_{l^{}_{\lceil\|{C^{}_{11;11}}\|\rceil}}\circ
{C^{}_{11;11}}(\cdot) \eqno{(2.33)}$$ is the invertible operator
$X^{}_{11;11}$ we need in $M^{}_{m^{}_{1}}(L^{\infty}(\mu^{}_{}))$.
In the same way, we obtain the invertible operators $X^{}_{11;22}$
and $X^{}_{11;33}$ for $C^{}_{11;22}$ and $C^{}_{11;33}$
respectively. Notice that the diagonal entries of $B^{}_{ii;st}$ are
the same for $1\leq i\leq 3$ and $1\leq s,t\leq m^{}_{i}$. Construct
an invertible operator $X$ in the commutant $\{UAU^{*}\}^{\prime}$
with $X^{}_{11;ii}$ for $1\leq i\leq 3$ such that
$XC^{\prime}X^{-1}$ is in the standard bounded maximal abelian set
of idempotents of $\{UAU^{*}\}^{\prime}$.
\end{proof}

\begin{lemma}
Let $\mathscr{P}$ be a bounded maximal abelian set of idempotents in
the commutant $\{A\}^{\prime}$, where $A$ is defined from $(1.13)$
to $(1.15)$. Then there exists a finite subset $\mathscr{P}^{}_{0}$
of $\mathscr{P}$ such that the equality
$$\mathscr{P}^{}_{0}(\lambda)=\mathscr{P}(\lambda)\eqno{(2.34)}$$
holds almost everywhere on $\sigma(N^{}_{\mu})$.
\end{lemma}

\begin{proof}
The motivation of this lemma is to find a Borel measurable skeleton
of $\mathscr {P}$.

By Lemma $2.2$, for an idempotent $P$ in $\mathscr{P}$, there exists
a unitary operator $U$ such that the operator $C=UPU^{*}$ is in the
form of $(2.22)$, and $C$ is similar to $C^{\prime}$ in
$\{UAU^{*}\}^{\prime}$, where $C^{\prime}$ is in the form of
$(2.28)$.

Let $E^{}_{i}$ be a projection in $\{UAU^{*}\}^{\prime}$, which is
as in the form of $(2.28)$
$$E^{}_{i}=\begin{pmatrix}
E^{}_{i;1}&0&\cdots&0\\
0&E^{}_{i;2}&\cdots&0\\
\vdots&\vdots&\ddots&\vdots\\
0&0&\cdots&E^{}_{i;n^{}_{1}}\\
\end{pmatrix}\mbox{  for } i=1,2,3, \eqno(2.35)$$
where, as in the form of $(2.27)$ we write $E^{}_{i;1}$ as a
$3$-by-$3$ block matrix, the $(i,i)$ block of $E^{}_{i;1}$ is the
identity of $M^{}_{m^{}_{i}}(L^{\infty}(\mu^{}_{}))$ and other
blocks are $0$, compared with $C^{\prime}_{11}$ in $(2.27)$. Thus
the projections $E^{}_{i;2},\ldots,E^{}_{i;n^{}_{1}}$ can be fixed
corresponding to $E^{}_{i;1}$. Therefore we have the equality
${E^{}_{i}}{C^{}_{}}{E^{}_{i}}=
{E^{}_{i}}{C^{\prime}_{}}{E^{}_{i}}$. Define a
$\mu^{}_{}$-measurable function $\mbox{r}^{}_{i}$ in the form
$$\mbox{r}^{}_{i}(P)(\lambda)\triangleq\dfrac{1}{n^{}_{i}}
\mbox{Tr}^{}_{{n^{}_{i}m^{}_{i}}}({E^{}_{i}}{UP^{}_{}U^{*}_{}}{E^{}_{i}}(\lambda)),
\mbox{ for almost every $\lambda$ in
$\sigma(N^{}_{\mu})$},\eqno(2.36)$$ where
$\mbox{Tr}^{}_{{n^{}_{i}m^{}_{i}}}$ denotes the standard trace on
$M^{}_{{n^{}_{i}m^{}_{i}}}(\mathbb {C})$.

We assert that there exists an idempotent $P$ in $\mathscr{P}$ such
that the inequality
$$0<\mbox{r}^{}_{1}(P)(\lambda)<m^{}_{1} \eqno(2.37)$$ holds almost
everywhere on $\sigma(N^{}_{\mu})$.

If $\mbox{r}^{}_{1}(P)(\lambda)=0$ or
$\mbox{r}^{}_{1}(P)(\lambda)=m^{}_{1}$ holds almost everywhere on
$\sigma(N^{}_{\mu})$ for every $P$ in $\mathscr{P}$, then
$\mathscr{P}$ is not bounded maximal abelian. Therefore, there
exists a subset $\Gamma^{}_{1}$ of $\sigma(N^{}_{\mu})$ with
$\mu^{}_{}(\Gamma^{}_{1})>0$ and an idempotent $P^{}_{1}$ in
$\mathscr{P}$ such that
$0<\mbox{r}^{}_{1}(P^{}_{1})(\lambda)<m^{}_{1}$ holds almost
everywhere on $\Gamma^{}_{1}$. In the same way, we have a subset
$\Gamma^{}_{2}$ of $\sigma(N^{}_{\mu})\backslash \Gamma^{}_{1}$ with
$\mu^{}_{}(\Gamma^{}_{2})>0$ and an idempotent $P^{}_{2}$ in
$\mathscr{P}$ such that
$0<\mbox{r}^{}_{1}(P^{}_{2})(\lambda)<m^{}_{1}$ holds almost
everywhere on $\Gamma^{}_{2}$. By Zorn lemma, there are sequences
$\{P^{}_{i}\}^{\infty}_{i=1}$ in $\mathscr{P}$ and
$\{\Gamma^{}_{i}\}^{\infty}_{i=1}$ with $\mu^{}_{}(\Gamma^{}_{i})>0$
for every $i$ and $\cup^{\infty}_{i=1}
{(\Gamma^{}_{i})}=\sigma(N^{}_{\mu})$ such that
$0<\mbox{r}^{}_{1}(P^{}_{i})(\lambda)<m^{}_{1}$ holds almost
everywhere on $\Gamma^{}_{i}$. Denote by $P$ the sum of the
restrictions of $P^{}_{i}$ on $\Gamma^{}_{i}$. Therefore, we obtain
the above assertion.

Next, we assert that there exists an idempotent $P$ in $\mathscr{P}$
such that the equality
$$\mbox{r}^{}_{1}(P)(\lambda)=1 \eqno(2.38)$$ holds almost
everywhere on $\sigma(N^{}_{\mu})$.

If $P$ is described as in the fist assertion, then
$\sigma(N^{}_{\mu})$ can be divided into at most $m^{}_{1}-1$
pairwise disjoint Borel subsets
$\{\Gamma^{}_{i}\}^{m^{}_{1}-1}_{i=1}$ corresponding to
$\mbox{r}^{}_{1}(P^{}_{})$ such that the equality
$\mbox{r}^{}_{1}(P)(\lambda)=i$ holds almost everywhere on
$\Gamma^{}_{i}$. Assume that $\mu^{}_{}(\Gamma^{}_{m^{}_{1}-1})>0$.
By a similar proof of the first assertion, there exists an
idempotent $P^{}_{1}$ in $\mathscr{P}$ such that the inequality
$0<\mbox{r}^{}_{1}(P^{}_{1})(\lambda)<m^{}_{1}-1$ holds almost
everywhere on $\Gamma^{}_{m^{}_{1}-1}$. Let $Q^{}_{1}$ denote the
sum of the restriction of $P^{}_{1}$ on $\Gamma^{}_{m^{}_{1}-1}$ and
the restriction of $P_{}$ on $\sigma(N^{}_{\mu})\backslash
\Gamma^{}_{m^{}_{1}-1}$. Redivide $\sigma(N^{}_{\mu})$ into at most
$m^{}_{1}-2$ pairwise disjoint Borel subsets
$\{\Gamma^{}_{i}\}^{m^{}_{1}-2}_{i=1}$ corresponding to
$\mbox{r}^{}_{}(Q^{}_{1})$ as above. Assume that
$\mu^{}_{}(\Gamma^{}_{m^{}_{1}-2})>0$. There exists an idempotent
$P^{}_{2}$ in $\mathscr{P}$ such that the inequality
$0<\mbox{r}^{}_{1}(P^{}_{2})(\lambda)<m^{}_{1}-2$ holds almost
everywhere on $\Gamma^{}_{m^{}_{1}-2}$. Construct $Q^{}_{2}$ with
$P^{}_{2}$ and $Q^{}_{1}$ as above. After at most ${m^{}_{1}-2}$
steps, we obtain an idempotent in $\mathscr{P}$ as required in the
second assertion.

Finally, we assert that there are $m^{}_{1}$ idempotents
$\{P^{}_{i}\}^{m^{}_{1}}_{i=1}$ in $\mathscr{P}$ such that the
equality
$$\mbox{r}^{}_{1}(P^{}_{i})(\lambda)=1 \eqno(2.39)$$ holds almost
everywhere on $\sigma(N^{}_{\mu})$, and $P^{}_{i}P^{}_{j}=0$ for
$i\neq j$.

By the second assertion, we obtain $P^{}_{1}$ in $\mathscr{P}$ such
that $\mbox{r}^{}_{1}(P^{}_{1})(\lambda)=1$ holds almost everywhere
on $\sigma(N^{}_{\mu})$. Then we obtain $P^{}_{2}$ in
$(I-P^{}_{1})\mathscr{P}$ such that
$\mbox{r}^{}_{1}(P^{}_{2})(\lambda)=1$ holds almost everywhere on
$\sigma(N^{}_{\mu})$ by applying the first two assertions. Take
these idempotents one by one and we prove the third assertion.

By the above three assertions, we obtain
$m^{}_{1}+m^{}_{2}+m^{}_{3}$ idempotents
$\{P^{}_{j;i}\}^{3;m^{}_{i}}_{i=1;j=1}$ in $\mathscr{P}$ such that
the equality
$$\mbox{r}^{}_{i}(P^{}_{j;i})(\lambda)=1 \eqno(2.40)$$ holds almost
everywhere on $\sigma(N^{}_{\mu})$, and $(P^{}_{j;i})
(P^{}_{l;k})=0$ for $i\neq k$ or $j\neq l$. Construct
$\mathscr{P}^{}_{0}$ in the form
$$\mathscr{P}^{}_{0}\triangleq\{\sum^{3}_{i=1}\sum^{m^{}_{i}}_{j=1}
\alpha^{}_{ij}
(P^{}_{j;i}):\alpha^{}_{ij}\in\{0,1\}\}.\eqno{(2.41)}$$ Then the
equality $\mathscr{P}^{}_{0}(\lambda)=\mathscr{P}(\lambda)$ holds
almost everywhere on $\sigma(N^{}_{\mu})$.
\end{proof}

\begin{proof}[Proof of Theorem $1.2$]
Let $\mathscr{P}$ be a bounded maximal abelian set of idempotents in
$\{A\}^{\prime}$. By Lemma $2.3$, there exist
$m^{}_{1}+m^{}_{2}+m^{}_{3}$ idempotents
$\{P^{}_{j;i}\}^{3;m^{}_{i}}_{i=1;j=1}$ in $\mathscr{P}$ such that
the equality $\mbox{r}^{}_{i}(P^{}_{j;i})(\lambda)=1$ holds almost
everywhere on $\sigma(N^{}_{\mu})$, and $P^{}_{j;i}P^{}_{l;k}=0$ for
$i\neq k$ or $j\neq l$. By Lemma $2.2$, there exists an invertible
operator $X^{}_{1;1}$ in $\{A\}^{\prime}$ such that
$X^{}_{1;1}P^{}_{1;1}X^{-1}_{1;1}$ is in the standard bounded
maximal abelian set of idempotents $\mathscr {E}$ in
$\{A\}^{\prime}$. Precisely, the idempotent
$X^{}_{1;1}P^{}_{1;1}X^{-1}_{1;1}$ is in the form
$$X^{}_{1;1}P^{}_{1;1}X^{-1}_{1;1}= (I\oplus
0^{(m^{}_{1}-1)}_{})\oplus(0^{(m^{}_{2})}_{})\oplus(0^{(m^{}_{3})}_{}),
\eqno(2.42)$$ where $I$ is the identity operator in
$M^{}_{n^{}_{1}}(L^{\infty}(\mu^{}_{}))$. In a similar way, there
exists an invertible operator $X^{}_{2;1}$ in $\{A\}^{\prime}$ such
that $(X^{}_{2;1}X^{}_{1;1})P^{}_{1;1}(X^{}_{2;1}X^{}_{1;1})^{-1}$
and $(X^{}_{2;1}X^{}_{1;1})P^{}_{2;1}(X^{}_{2;1}X^{}_{1;1})^{-1}$
are both in the standard bounded maximal abelian set of idempotents
in $\{A\}^{\prime}$. The invertible operator $X^{}_{2;1}$ is in the
form
$$X^{}_{2;1}=
\begin{pmatrix}
I&0\\
0&*\\
\end{pmatrix}, \eqno(2.43)$$ where $I$ is the identity operator in
$M^{}_{n^{}_{1}}(L^{\infty}(\mu^{}_{}))$. Furthermore, there exist
$m^{}_{1}+m^{}_{2}+m^{}_{3}-3$ invertible operators
$\{X^{}_{j;i}\}^{3;m^{}_{i}-1}_{i=1;j=1}$ in $\{A\}^{\prime}$ such
that $X(P^{}_{j;i})X^{-1}$  is in the standard bounded maximal
abelian set of idempotents in $\{A\}^{\prime}$ for every $i$ and
$j$, where let $X$ denote the product
$$X=X^{}_{m^{}_{3}-1;3}\cdots X^{}_{1;3} X^{}_{m^{}_{2}-1;2}\cdots
X^{}_{1;2} X^{}_{m^{}_{1}-1;1}\cdots X^{}_{1;1}. \eqno(2.44)$$ Then
we obtain that the set $X\mathscr{P}X^{-1}$ is the standard bounded
maximal abelian set of idempotents in the commutant
$\{A\}^{\prime}$. Therefore, the strongly irreducible decomposition
of $A$ is unique up to similarity.

Next, we compute the $K^{}_{}$ groups of $\{A\}^{\prime}$. We denote
by $\mathscr {J}$ a closed two-sided ideal of $\{A\}^{\prime}$ such
that for every operator $B$ in $\mathscr {J}$, every entry in the
main diagonal of $B^{}_{ii;st}$ is $0$ for $1\leq i\leq 3$ and
$1\leq s,t\leq m^{}_{i}$, where $B$ and $B^{}_{ii;st}$ are as in the
form of $(2.20)$ and $(2.21)$. By $\mathscr {B}$ we denote a
subalgebra of $\{A\}^{\prime}$ such that for every operator $B$ in
$\mathscr {B}$, every entry of $B^{}_{ij;st}$ is $0$ except ones in
the main diagonal of $B^{}_{ii;st}$, for $1\leq i,j\leq 3$ and
$1\leq s,t\leq m^{}_{i}$. By observation, we obtain the following
split short exact sequence:
$$\xymatrix@C=0.5cm{
0\ar[r] &\mathscr {J}\ar[r]^-{\iota} &\{A\}^{\prime}\ar@<0.5ex>[r]
^-{\pi}&\mathscr {B}\ar@<0.5ex>[l]^-{\alpha}\ar[r]&0}
\eqno{(2.45)}$$ where we denote by $\iota$ and $\alpha$ the
inclusion maps and by $\pi$ the map such that for every operator $B$
in $\{A\}^{\prime}$, every entry of $\pi(B)^{}_{ij;st}$ is $0$
except ones in the main diagonal of $B^{}_{ii;st}$ staying invariant
with respect to $\pi$, for $1\leq i,j\leq 3$ and $1\leq s,t\leq
m^{}_{i}$. Essentially, $\pi$ is the quotient map. Furthermore, we
obtain
$$\mathscr {B}\cong M^{}_{m^{}_{1}}(L^{\infty}(\mu))
\oplus M^{}_{m^{}_{2}}(L^{\infty}(\mu)) \oplus
M^{}_{m^{}_{3}}(L^{\infty}(\mu)). \eqno{(2.46)}$$

By Lemma $2.2$, we have $K^{}_{0}(\pi)$ is an isomorphism.
Therefore, $$K^{}_{0}(\{A\}^{\prime})\cong K^{}_{0}(\mathscr {B})
\eqno{(2.47)}$$ and by a routine computation, we obtain
$$K^{}_{0}(\{A\}^{\prime})\cong
\{f:\sigma(N^{}_{\mu})\rightarrow\mathbb{Z}^{(3)}_{}, f\mbox{ is
bounded Borel}\}. \eqno(2.48)$$

For a generalized case, we need to combine the proofs as above with
respect to different regular Borel measures which are pairwise
mutually singular. Since the spectrum of $A^{}_{nm}$ (as in $(1.10)$
and $(1.11)$) equals $\sigma(N^{}_{\nu^{}_{nm}})$, we construct a
normal operator
$$N=\bigoplus^{\infty}_{n=1}\bigoplus^{\infty}_{m=1}N^{}_{\nu^{}_{nm}}, \eqno(2.49)$$
where $N^{}_{\nu^{}_{nm}}=0$ holds for all but finitely many $n$ and
$m$ in $\mathbb{N}$, corresponding to the assumption from $(1.10)$
to $(1.12)$. We observe that for $i\neq j$, the scalar-valued
spectral measures $\nu^{}_{nm^{}_{i}}$ and $\nu^{}_{nm^{}_{j}}$ are
mutually singular, but the scalar-valued spectral measures
$\nu^{}_{n^{}_{i}m}$ and $\nu^{}_{n^{}_{j}m}$ may not be mutually
singular. By (\cite{Conway_1}, IX, Theorem $10.16$), the normal
operator $N$ can be expressed in a direct sum of finitely many
normal operators with pairwise mutually singular scalar-valued
spectral measures. Actually, this is a finer decomposition than the
one in $(2.49)$. With this expression and the above proof for a
special case, we obtain the proof of Theorem $1.2$ and Theorem
$1.3$.
\end{proof}

By Theorem $1.2$, we can compute the $K^{}_{0}$ group of
$\{A\}^{\prime}$, if the strongly irreducible decomposition of $A$
is unique up to similarity. Next, we investigate the uniqueness of
the strongly irreducible decomposition of $A$ up to similarity by
the $K^{}_{0}$ group of $\{A\}^{\prime}$. Let operators $A$ and
$B^{}_{}$ be as in the form of $(1.14)$ and $(1.15)$:
$$A^{}_{}=\left(\begin{array}{cccccc}
N^{}_{\mu^{}_{}}&M^{}_{f^{}_{12}}&\cdots&M^{}_{f^{}_{1n}}\\
0&N^{}_{\mu^{}_{}}&\cdots&M^{}_{f^{}_{2n}}\\
\vdots&\vdots&\ddots&\vdots\\
0&0&\cdots&N^{}_{\mu^{}_{}}\\
\end{array}\right)^{}_{n\times n} \eqno{(2.50)}$$ and
$$B^{}_{}=\left(\begin{array}{cccccc}
N^{}_{\mu^{}_{}}&M^{}_{g^{}_{12}}&\cdots&M^{}_{g^{}_{1n}}\\
0&N^{}_{\mu^{}_{}}&\cdots&M^{}_{g^{}_{2n}}\\
\vdots&\vdots&\ddots&\vdots\\
0&0&\cdots&N^{}_{\mu^{}_{}}\\
\end{array}\right)^{}_{n\times n}, \eqno(2.51)$$ where
$M^{}_{f^{}_{ij}}$ and $M^{}_{g^{}_{ij}}$ is in
$\{N^{(n)}_{\mu^{}_{}}\}^{\prime}$, for $1\leq i<j\leq n$. Then we
have the following lemma.

\begin{lemma}
The operators $A^{(m^{}_{1})}_{}$ and $B^{(m^{}_{2})}_{}$ are
similar in $M^{}_{nm^{}_{1}}(L^{\infty}(\mu^{}_{}))$ $(m^{}_{1}\geq
m^{}_{2})$ if and only if there exists an isomorphism $\theta$ such
that
$$\left.\begin{array}{ll}
(1)&K^{}_{0}(\{T\}^{\prime})\cong\{f:\sigma(N^{}_{\mu})\rightarrow\mathbb{Z}^{}_{},f
\mbox{ is bounded Borel}\}, \mbox{ and}\\
(2)&\theta([I^{}_{\{T\}^{\prime}}])=2m^{}_{1}e^{}_{},
\end{array}\right\}\eqno{(2.52)}$$
 where $T=A^{(m^{}_{1})}_{}\oplus B^{(m^{}_{2})}_{}$ and
$e(\lambda)$ is the generator of the semigroup $\mathbb {N}$ of
$\mathbb {Z}$ for almost every $\lambda$ in $\sigma(N^{}_{\mu})$.
\end{lemma}

\begin{proof}
If the operators $A^{(m^{}_{1})}_{}$ and $B^{(m^{}_{2})}_{}$ are
similar in $M^{}_{nm^{}_{1}}(L^{\infty}(\mu^{}_{}))$, then we obtain
$K^{}_{0}(\{T\}^{\prime}_{})$ as required by the proof of Theorem
$1.2$.

On the other hand, we suppose that the relations in $(2.52)$ hold.
Let $P$ and $Q$ be idempotents in
$\{A^{(m^{}_{1})}_{}\}^{\prime}_{}$ and
$\{B^{(m^{}_{2})}_{}\}^{\prime}_{}$ respectively such that the
equalities
$$\mbox{r}^{}_{\{A^{(m^{}_{1})}_{}\}^{\prime}_{}}(P)(\lambda)=1\quad
\mbox{ and }\quad
\mbox{r}^{}_{\{B^{(m^{}_{2})}_{}\}^{\prime}_{}}(Q)(\lambda)=1
\eqno{(2.53)}$$ hold for almost every $\lambda$ in
$\sigma(N^{}_{\mu})$. If $P\oplus 0$ and $0\oplus Q$ are not similar
in $\{T\}^{\prime}_{}$, then we obtain $\theta([P\oplus
0])=e^{}_{}=\theta([0\oplus Q])$. Thus $\theta$ is not an
isomorphism which contradicts the assumption in $(2.52)$. Therefore
$P\oplus 0$ and $0\oplus Q$ are similar in $\{T\}^{\prime}_{}$. We
can choose projections $E$ and $F$ similar to $P\oplus 0$ in
$\{T\}^{\prime}_{}$ such that $T|^{}_{{\mathrm{ran}}E}=A^{}_{}$ and
$T|^{}_{{\mathrm{ran}}F}=B^{}_{}$. Thus $A^{}_{}\oplus 0$ is similar
to $0\oplus B^{}_{}$ in $\{T\}^{\prime}_{}$. The equality
$\theta([I^{}_{\{T\}^{\prime}}])=2m^{}_{1}e^{}_{1}$ yields that
$m^{}_{1}+m^{}_{2}=2m^{}_{1}$. Hence $m^{}_{1}=m^{}_{2}$ and
$A^{(m^{}_{1})}_{}$ is similar to $B^{(m^{}_{2})}_{}$.
\end{proof}

\begin{proof}[Proof of Theorem $1.4$]
If the operator $A=\oplus^{3}_{i=1} A^{(m^{}_{i})}_{n^{}_{i}}$ is
similar to $B=\oplus^{3}_{j=1} B^{(k^{}_{j})}_{l^{}_{j}}$, then we
can obtain an isomorphism $\theta$ and the $K^{}_{0}$ group
$K^{}_{0}(\{T\}^{\prime})$ as required in the theorem by a routine
computation.

To show the converse, suppose that there exists an isomorphism
$\theta$ such that
\begin{enumerate}
\item [(a)] $\theta:K^{}_{0}(\{T\}^{\prime})\rightarrow
\{f:\sigma(N^{}_{\mu})\rightarrow\mathbb {Z}^{(3)}_{}, f\mbox{ is
bounded Borel}\}$ and
\item [(b)] $\theta([I^{}_{\{T\}^{\prime}}])=
2m^{}_{1}e^{}_{1}+2m^{}_{2}e^{}_{2}+2m^{}_{3}e^{}_{3}$.
\end{enumerate}
In the commutant $\{T\}^{\prime}$, there exist $3$ projections
$\{E^{}_{i}\}^{3}_{i=1}$ and $3$ projections
$\{F^{}_{j}\}^{3}_{j=1}$ such that
\begin{enumerate}
\item [(1)] $T|^{}_{\mathrm{ran} E^{}_{i}}=A^{}_{n^{}_{i}}$
and $T|^{}_{\mathrm{ran} F^{}_{j}}=B^{}_{l^{}_{j}}$;
\item [(2)] $E^{}_{i}E^{}_{j}=F^{}_{i}F^{}_{j}=0$ and
$E^{}_{i}F^{}_{j}=0$ for $i\neq j$;
\item [(3)] the equalities
$\mbox{r}^{}_{i}(E^{}_{i})(\lambda)=1$ and
$\mbox{r}^{}_{j}(F^{}_{j})(\lambda)=1$ hold for almost every
$\lambda$ in $\sigma(N^{}_{\mu})$ and $1\leq i,j\leq 3$.
\end{enumerate}
The equivalence classes $\{[E^{}_{i}]\}^{3}_{i=1}$ can be considered
as the generating set of $K^{}_{0}(\{T\}^{\prime})$. If $F^{}_{i}$
is not similar to $E^{}_{i}$ in $\{T\}^{\prime}$ for some $i$, then
for $K^{}_{0}(\{T\}^{\prime})$, there exists a $\lambda$ in the
$\sigma(N^{}_{\mu})$ such that the set
$\{E^{}_{j}(\lambda)\}^{3}_{j=1}\cup\{F^{}_{i}(\lambda)\}$ generates
$\mathbb {Z}^{(4)}$, which is a contradiction since $\lambda$ can
not be removed from $\sigma(N^{}_{\mu})$. Therefore, $F^{}_{i}$ is
similar to $E^{}_{i}$ in $\{T\}^{\prime}$ for $1\leq i\leq 3$. The
coefficient of $e^{}_{i}$ in $\theta([I^{}_{\{T\}^{\prime}}])$ is
$m^{}_{i}+k^{}_{i}=2m^{}_{i}$ for $1\leq i\leq 3$. Therefore the
equality $m^{}_{i}=k^{}_{i}$ holds for $1\leq i\leq 3$. Thus we
obtain that the operator $A$ is similar to $B$.
\end{proof}

\section{Appendix}
In this part, we show the relation between an operator $A$ as in
$(1.16)$ and $B^{(m)}_{}$ such that $B$ is as in the form of
$(2.51)$ and $m$ is a positive integer. The motivation is to obtain
a decomposition of an operator $A$ as in $(1.16)$ with respect to
the main diagonal entries. Suppose that $A$ is an operator in the
form
$$A=\begin{pmatrix}
M^{}_{f^{}_{}}&M^{}_{f^{}_{12}}&\cdots&M^{}_{f^{}_{1n}}\\
0&M^{}_{f^{}_{}}&\cdots&M^{}_{f^{}_{2n}}\\
\vdots&\vdots&\ddots&\vdots\\
0&0&\cdots&M^{}_{f^{}_{}}\\
\end{pmatrix}^{}_{n\times n}
\begin{matrix}
L^{2}(\mu)\\
L^{2}(\mu)\\
\vdots\\
L^{2}(\mu)\\
\end{matrix}, \eqno{(3.1)}$$
and there exists a unitary operator $V$ such that
$VM^{}_{f}V^{*}=N^{(m)}_{\nu^{}_{}}$ where $f$ and $f^{}_{ij}$ are
in $L^{\infty}(\mu^{}_{})$ and $\nu^{}_{}=\mu^{}_{}\circ f^{-1}_{}$.
Then we have the following proposition.
\begin{proposition}
There is a unitary operator $W$ such that
$WM^{}_{f}W^{*}=N^{(m)}_{\nu^{}_{}}$ and
$$W^{(n)}_{}A(W^{*}_{})^{(n)}_{}=\bigoplus^{m}_{k=1}\begin{pmatrix}
N^{}_{\nu^{}_{}}&M^{}_{f^{}_{k;12}}&M^{}_{f^{}_{k;13}}&\cdots&M^{}_{f^{}_{k;1n}}\\
0&N^{}_{\nu^{}_{}}&M^{}_{f^{}_{k;23}}&\cdots&M^{}_{f^{}_{k;2n}}\\
0&0&N^{}_{\nu^{}_{}}&\cdots&M^{}_{f^{}_{k;3n}}\\
\vdots&\vdots&\vdots&\ddots&\vdots\\
0&0&0&\cdots&N^{}_{\nu^{}_{}}\\
\end{pmatrix}^{}_{n\times n}. \eqno(3.2)$$
\end{proposition}

\begin{proof}
The multiplication operators $M^{}_{f^{}_{}}$ and $M^{}_{f^{}_{ij}}$
are in $\{N^{}_{\mu^{}_{}}\}^{\prime}$ for $i<j$. By the assumption
in $(3.1)$, there exists a unitary operator $V$ such that
$VM^{}_{f^{}_{n}}V^{*}=N^{(m)}_{\nu^{}_{}}$. Then we obtain the
equality
$$N^{(m)}_{\nu^{}_{}}(VM^{}_{f^{}_{ij}}V^{*})
=(VM^{}_{f^{}_{ij}}V^{*})N^{(m)}_{\nu^{}_{}}.\eqno(3.3)$$ Therefore,
the operator $VM^{}_{f^{}_{ij}}V^{*}$ can be expressed in the form
$$VM^{}_{f^{}_{ij}}V^{*}=
\begin{pmatrix}
M^{}_{\phi^{}_{ij;11}}&\cdots&M^{}_{\phi^{}_{ij;1m}}\\
\vdots&\ddots&\vdots\\
M^{}_{\phi^{}_{ij;m1}}&\cdots&M^{}_{\phi^{}_{ij;mm}}\\
\end{pmatrix}^{}_{m\times m}, \eqno(3.4)$$
where $\phi^{}_{ij;st}$ is in $L^{\infty}(\nu^{}_{})$ and
$M^{}_{\phi^{}_{ij;st}}$ is in $\{N^{}_{\nu^{}_{}}\}^{\prime}$ for
$1\leq s,t\leq m$. The motivation is to find a unitary operator such
that every $VM^{}_{f^{}_{ij}}V^{*}$ is unitarily equivalent to a
diagonal operator in the commutant
$\{N^{(m)}_{\nu^{}_{}}\}^{\prime}$. We observe that
$$\{VN^{}_{\mu^{}_{}}V^{*}\}^{\prime\prime}(=\{VN^{}_{\mu^{}_{}}V^{*}\}^{\prime})
\subseteq\{N^{(m)}_{\nu^{}_{}}\}^{\prime}. \eqno{(3.5)}$$ Let
$\mathscr{E}$ be the set of projections in
$\{VN^{}_{\mu^{}_{}}V^{*}\}^{\prime\prime}$. Then $\mathscr{E}$ is a
maximal abelian set of projections in
$\{N^{(m)}_{\nu^{}_{}}\}^{\prime}$. By (\cite{Shi_1}, Proposition
$4.1$), we obtain that $\mathscr{E}$ is a bounded maximal abelian
set of idempotents in $\{N^{(m)}_{\nu^{}_{}}\}^{\prime}$. As an
application of Lemma $2.3$, there exist $m$ projections
$\{E^{}_{i}\}^{m}_{i=1}$ in $\mathscr{E}$ such that
$E^{}_{i}E^{}_{j}=0$ for $i\neq j$ and the equality
$\mbox{rank}(E^{}_{i}(\lambda))=1$ holds for almost every $\lambda$
in the support of $\nu^{}_{}$. Then by (\cite{Azoff_1},Corollary
$2$) and a similar proof of Theorem $1.2$, we obtain a unitary
operator $V^{}_{1}$ in $\{N^{(m)}_{\nu^{}_{}}\}^{\prime}$ such that
every projection in $V^{}_{1}\mathscr{E}V^{*}_{1}$ is diagonal.
Therefore $(V^{}_{1}V^{}_{})^{(n)}_{}$ is as required.
\end{proof}

By Proposition $3.1$, we observe that $B^{(m^{}_{})}_{}$ for $B$ as
in $(2.51)$ is a special form of $(3.2)$. When we consider a similar
result as Theorem $1.2$, the following example makes the calculation
appear to be more complicated. Let the operators $X$ and $Y$ be in
the form
$$X=\begin{pmatrix}
N^{}_{\mu^{}_{}}&I\\
0&N^{}_{\mu^{}_{}}\\
\end{pmatrix},\quad
Y=\begin{pmatrix}
N^{}_{\mu^{}_{}}&N^{}_{\mu^{}_{}}\\
0&N^{}_{\mu^{}_{}}\\
\end{pmatrix}, \eqno(3.6)$$
Then $X$ and $Y$ are not similar in
$M^{}_{2}(L^{\infty}(\mu^{}_{}))$, where the regular Borel measure
$\mu$ is supported on the interval $[-1,1]$. By Lemma $2.1$, if $Z$
is a bounded linear operator such that $XZ=ZY$, then $Z$ is in the
form
$$Z=\begin{pmatrix}
M^{}_{f^{}_{1}}&M^{}_{f^{}_{12}}\\
0&M^{}_{f^{}_{2}}\\
\end{pmatrix}, \eqno(3.7)$$
where every entry of $Z$ is in $\{N^{}_{\mu^{}_{}}\}^{\prime}$. And
$M^{}_{f^{}_{2}}=M^{}_{f^{}_{1}}N^{}_{\mu^{}_{}}$. Therefore, the
operator $Z$ is not invertible. In (\cite{Deckard_1}, \S $2$), a
similar example was provided. Define $T=X\oplus Y$. Let $E$ be the
projection $I\oplus 0$ such that $T|^{}_{{\rm{ran}}E}=X$ and $F$ be
the projection $0\oplus I$ such that $T|^{}_{{\rm{ran}}F}=Y$. Then
$E$ is not similar to $F$ in $\{T\}^{\prime}$ corresponding to the
above discussion. This operator is different from the operator
investigated in (\cite{Shi_2}, Theorem $3.3$).

In the following we show that the multiplicity ``$\infty$'' is not
what we want in the decomposition of an operator as in the form of
$(1.16)$.

\begin{proposition}
Let $A^{}_{}$ be an operator as in $(2.50)$
$$A^{(\infty)}_{}=\begin{pmatrix}
N^{(\infty)}_{\mu^{}_{}}&M^{(\infty)}_{f^{}_{12}}&M^{(\infty)}_{f^{}_{13}}&\cdots&M^{(\infty)}_{f^{}_{1n}}\\
0&N^{(\infty)}_{\mu^{}_{}}&M^{(\infty)}_{f^{}_{23}}&\cdots&M^{(\infty)}_{f^{}_{2n}}\\
0&0&N^{(\infty)}_{\mu^{}_{}}&\cdots&M^{(\infty)}_{f^{}_{3n}}\\
\vdots&\vdots&\vdots&\ddots&\vdots\\
0&0&0&\cdots&N^{(\infty)}_{\mu^{}_{}}\\
\end{pmatrix}_{n\times n}, \eqno(3.8)$$
where $\mu$ is stated as in $(1.7)$. Then the strongly irreducible
decomposition of $A^{(\infty)}_{}$ is not unique up to similarity.
\end{proposition}

\begin{proof}
We need to construct two bounded maximal abelian sets of idempotents
in $\{A^{(\infty)}_{}\}^{\prime}$ such that they are not similar to
each other.

We write $N^{(\infty)}_{\mu^{}_{}}$ in the form
$N_{\mu^{}_{}}\otimes I^{}_{l^{2}}$, where $I^{}_{l^{2}}$ is the
identity operator on $l^{2}$. Denote by $\mathscr{P}$ the set of all
the spectral projections of $N^{}_{\mu^{}_{}}$. This set forms a
bounded maximal abelian set of idempotents in
$\{N^{}_{\mu^{}_{}}\}^{\prime}$. Let $\{e^{}_{k}\}^{\infty}_{k=1}$
be an orthonormal basis for $l^{2}$. Denote by $E^{}_{k}$ the
projection such that $\mbox{ran}E^{}_{k}=\{\lambda
e^{}_{k}:\lambda\in\mathbb{C}\}$. Define $$\mathscr
{Q}^{}_{1}\triangleq\{P\in\mathscr{L}(l^{2}):P=P^{*}=P^{2}
\in\{E^{}_{k}:k\in\mathbb{N}\}^{\prime\prime}\}. \eqno(3.9)$$ Denote
by $\chi^{}_{S}$ the characteristic function for a Borel subset $S$
in $\sigma(N^{}_{\mu})$ and define $$\hat{\mathscr
{Q}}^{}_{2}\triangleq\{M_{\chi^{}_{S}}\in\mathscr{L}(L^{2}(\mu)):
S\subseteq\sigma(N^{}_{\mu})\mbox{ is Borel}\}. \eqno(3.10)$$ There
is a unitary operator $U:L^{2}[0,1]\rightarrow l^{2}$ such that
$UPU^{*}\in\mathscr{L}(l^{2})$ for every $P\in\hat{\mathscr
{Q}}^{}_{2}$. The sets ${\mathscr {Q}}^{}_{2}\triangleq
U\hat{\mathscr {Q}}^{}_{2}U^{*}$ and $\mathscr {Q}^{}_{1}$ are two
bounded maximal abelian sets of idempotents in $\mathscr{L}(l^{2})$
but they are not unitarily equivalent.

The fact that $W^{*}(\mathscr {P})\otimes W^{*}(\mathscr
{Q}^{}_{1})$ and $W^{*}(\mathscr {P})\otimes W^{*}(\mathscr
{Q}^{}_{2})$ are both maximal abelian von Neumann algebras yields
that
$$\mathscr {F}^{}_{1}\triangleq\{P\in W^{*}(\mathscr {P})\otimes
W^{*}(\mathscr {Q}^{}_{1}):P=P^{*}=P^{2}\} \eqno(3.11)$$ and
$$\mathscr {F}^{}_{2}\triangleq\{P\in W^{*}(\mathscr {P})\otimes
W^{*}(\mathscr {Q}^{}_{2}):P=P^{*}=P^{2}\} \eqno(3.12)$$ are both
bounded maximal abelian sets of idempotents in
$\{N_{\mu^{}_{}}\otimes
I^{}_{l^{2}}\}^{\prime}=L^{\infty}(\mu^{}_{})\otimes\mathscr{L}(l^{2})$.

We assert that $\mathscr {F}^{(n)}_{i}$ is a bounded maximal abelian
set of idempotents in $\{A^{(\infty)}_{}\}^{\prime}$ for $i=1,2$.

An operator $X$ in $\{A^{(\infty)}_{}\}^{\prime}$ can be expressed
in the form
$$X=\begin{pmatrix}
X^{}_{11}&X^{}_{12}&X^{}_{13}&\cdots&X^{}_{1n}\\
X^{}_{21}&X^{}_{22}&X^{}_{23}&\cdots&X^{}_{2n}\\
X^{}_{31}&X^{}_{32}&X^{}_{33}&\cdots&X^{}_{3n}\\
\vdots&\vdots&\vdots&\ddots&\vdots\\
X^{}_{n1}&X^{}_{n2}&X^{}_{n3}&\cdots&X^{}_{nn}\\
\end{pmatrix}^{}_{n\times n}. \eqno (3.13)$$ By a similar proof of
Lemma $2.1$, we obtain that $X^{}_{ij}$ is in
$\{N_{\mu^{}_{}}\otimes I^{}_{l^{2}}\}^{\prime}$ and the equation
$X^{}_{ij}=0$ holds for $i>j$ and $X^{}_{ii}=X^{}_{11}$ for
$i=2,\ldots,n$ in $(3.13)$. Furthermore, if $X$ as in $(3.13)$ is an
idempotent, then so is every main diagonal entry $X^{}_{ii}$ of $X$.

We assume that $X$ is an idempotent in
$\{A^{(\infty)}_{}\}^{\prime}$ and commutes with $\mathscr
{F}^{(n)}_{1}$. Hence $X^{}_{ii}$ commutes with $\mathscr
{F}^{}_{1}$. The fact that $\mathscr {F}^{}_{1}$ is a maximal
abelian set of idempotents implies that $X^{}_{ii}$ belongs to
$\mathscr {F}^{}_{1}$. Thus $X^{}_{ii}$ commutes with $X^{}_{ij}$.
For the $1$-diagonal entries, the equation
$2X^{}_{ii}X^{}_{i,i+1}-X^{}_{i,i+1}=0$ yields $X^{}_{i,i+1}=0$, for
$i=1,\ldots,n-1$. By this way, the $k$-diagonal entries of $X$ are
all zero, for $k=2,\ldots,n$. Therefore $X$ is in $\mathscr
{F}^{(n)}_{1}$. Both $\mathscr {F}^{(n)}_{1}$ and $\mathscr
{F}^{(n)}_{2}$ are bounded maximal abelian sets of idempotents in
$\{A^{(\infty)}_{n}\}^{\prime}$.

Every operator $X$ in $\{A^{(\infty)}_{n}\}^{\prime}$ can be
expressed in the form
$$X=\int_{\sigma(N^{}_{\mu^{}_{}})} X(\lambda) d\mu^{}_{}(\lambda). \eqno(3.14)$$
Suppose that there is an invertible operator $X$ in
$\{A^{(\infty)}_{n}\}^{\prime}$ such that $$X\mathscr
{F}^{(n)}_{2}X^{-1}=\mathscr {F}^{(n)}_{1}.$$ For each $P$ in
$\mathscr {F}^{(n)}_{2}$, the projection $P(\lambda)$ is either of
rank $\infty$ or $0$, for almost every $\lambda$ in
$\sigma(N^{}_{\mu^{}_{}})$. But there exists an projection $Q$ in
$\mathscr {F}^{(n)}_{1}$ such that $Q(\lambda)$ is of rank $n$, for
almost every $\lambda$ in $\sigma(N^{}_{\mu^{}_{}})$. This is a
contradiction. Therefore $\mathscr {F}^{(n)}_{1}$ and $\mathscr
{F}^{(n)}_{2}$ are not similar in $\{A^{(\infty)}_{}\}^{\prime}$.
\end{proof}

As an application of the preceding proposition, we obtain the
following corollary.

\begin{corollary}
Let $A$ be an operator assumed as in $(1.16)$. If the multiplicity
function $m^{}_{{{f^{}_{}}}}$ of the main diagonal operator
$M^{}_{f}$ takes finitely many values and there exists a bounded
$\mathbb{N}$-valued simple function $r^{}_{A}$ on $\sigma(A)$ such
that
$$K^{}_{0}(\{A\}^{\prime})\cong\{\phi(\lambda)\in\mathbb
{Z}^{(r^{}_{A}(\lambda))}:\phi\ {\mbox{is\ bounded Borel on
}}\sigma(A)\}, \eqno{(3.15)}$$ then the strongly irreducible
decomposition of $A$ is unique up to similarity.
\end{corollary}

\bibliographystyle{amsplain}

\end{document}